\documentclass[12pt]{amsart}
 \usepackage[margin=1in]{geometry}
 \usepackage{amsmath,amsfonts,amsthm,amssymb,bbm}
 \usepackage{graphicx,color,dsfont}
 \usepackage{enumitem}
 \usepackage{fourier}
  \usepackage{hyperref}

\begin{document}

\newtheorem{theorem}{Theorem}
\newtheorem{lemma}{Lemma}
\newtheorem{proposition}{Proposition}
\newtheorem{rmk}{Remark}
\newtheorem{example}{Example}
\newtheorem{exercise}{Exercise}
\newtheorem{definition}{Definition}
\newtheorem{corollary}{Corollary}
\newtheorem{notation}{Notation}
\newtheorem{claim}{Claim}

\newtheorem{dif}{Definition}

 \newtheorem{thm}{Theorem}[section]
 \newtheorem{cor}[thm]{Corollary}
 \newtheorem{lem}[thm]{Lemma}
 \newtheorem{prop}[thm]{Proposition}
 \theoremstyle{definition}
 \newtheorem{defn}[thm]{Definition}
 \theoremstyle{remark}
 \newtheorem{rem}[thm]{Remark}
 \newtheorem*{ex}{Example}
 \numberwithin{equation}{section}

\newcommand{\vertiii}[1]{{\left\vert\kern-0.25ex\left\vert\kern-0.25ex\left\vert #1
    \right\vert\kern-0.25ex\right\vert\kern-0.25ex\right\vert}}

\newcommand{\R}{{\mathbb R}}
\newcommand{\C}{{\mathbb C}}
\newcommand{\U}{{\mathcal U}}
\newcommand{\cG}{W}
\newcommand{\tcG}{\tilde{\cG}}

\newcommand{\norm}[1]{\left\|#1\right\|}
\renewcommand{\(}{\left(}
\renewcommand{\)}{\right)}
\renewcommand{\[}{\left[}
\renewcommand{\]}{\right]}
\newcommand{\f}[2]{\frac{#1}{#2}}
\newcommand{\im}{i}
\newcommand{\cl}{{\mathcal L}}
\newcommand{\ck}{{\mathcal K}}

\newcommand{\al}{\alpha}
\newcommand{\be}{\beta}
\newcommand{\wh}[1]{\widehat{#1}}
\newcommand{\ga}{\gamma}
\newcommand{\Ga}{\Gamma}
\newcommand{\de}{\delta}
\newcommand{\ben}{\beta_n}
\newcommand{\De}{\Delta}
\newcommand{\ve}{\varepsilon}
\newcommand{\ze}{\zeta}
\newcommand{\Th}{\Theta}
\newcommand{\ka}{\kappa}
\newcommand{\la}{\lambda}
\newcommand{\laj}{\lambda_j}
\newcommand{\lak}{\lambda_k}
\newcommand{\La}{\Lambda}
\newcommand{\si}{\sigma}
\newcommand{\Si}{\Sigma}
\newcommand{\vp}{\varphi}
\newcommand{\om}{\omega}
\newcommand{\Om}{\Omega}
\newcommand{\ra}{\rightarrow}

\newcommand{\ro}{{\mathbf R}}
\newcommand{\rn}{{\mathbf R}^n}
\newcommand{\rd}{{\mathbf R}^d}
\newcommand{\rmm}{{\mathbf R}^m}
\newcommand{\rone}{\mathbb R}
\newcommand{\rtwo}{\mathbf R^2}
\newcommand{\rthree}{\mathbf R^3}
\newcommand{\rfour}{\mathbf R^4}
\newcommand{\ronen}{{\mathbf R}^{n+1}}
\newcommand{\ku}{\mathbf u}
\newcommand{\kw}{\mathbf w}
\newcommand{\kf}{\mathbf f}
\newcommand{\kz}{\mathbf z}

\newcommand{\N}{\mathbf N}

\newcommand{\tn}{\mathbf T^n}
\newcommand{\tone}{\mathbf T^1}
\newcommand{\ttwo}{\mathbf T^2}
\newcommand{\tthree}{\mathbf T^3}
\newcommand{\tfour}{\mathbf T^4}

\newcommand{\zn}{\mathbf Z^n}
\newcommand{\zp}{\mathbf Z^+}
\newcommand{\zone}{\mathbf Z^1}
\newcommand{\zz}{\mathbf Z}
\newcommand{\ztwo}{\mathbf Z^2}
\newcommand{\zthree}{\mathbf Z^3}
\newcommand{\zfour}{\mathbf Z^4}

\newcommand{\hn}{\mathbf H^n}
\newcommand{\hone}{\mathbf H^1}
\newcommand{\htwo}{\mathbf H^2}
\newcommand{\hthree}{\mathbf H^3}
\newcommand{\hfour}{\mathbf H^4}

\newcommand{\cone}{\mathbf C^1}
\newcommand{\ctwo}{\mathbf C^2}
\newcommand{\cthree}{\mathbf C^3}
\newcommand{\cfour}{\mathbf C^4}
\newcommand{\dpr}[2]{\langle #1,#2 \rangle}

\newcommand{\sn}{\mathbf S^{n-1}}
\newcommand{\sone}{\mathbf S^1}
\newcommand{\stwo}{\mathbf S^2}
\newcommand{\sthree}{\mathbf S^3}
\newcommand{\sfour}{\mathbf S^4}

\newcommand{\lp}{L^{p}}
\newcommand{\lppr}{L^{p'}}
\newcommand{\lqq}{L^{q}}
\newcommand{\lr}{L^{r}}
\newcommand{\echi}{(1-\chi(x/M))}
\newcommand{\chip}{\chi'(x/M)}

\newcommand{\wlp}{L^{p,\infty}}
\newcommand{\wlq}{L^{q,\infty}}
\newcommand{\wlr}{L^{r,\infty}}
\newcommand{\wlo}{L^{1,\infty}}

\newcommand{\lprn}{L^{p}(\rn)}
\newcommand{\lptn}{L^{p}(\tn)}
\newcommand{\lpzn}{L^{p}(\zn)}
\newcommand{\lpcn}{L^{p}(\cn)}
\newcommand{\lphn}{L^{p}(\cn)}

\newcommand{\lprone}{L^{p}(\rone)}
\newcommand{\lptone}{L^{p}(\tone)}
\newcommand{\lpzone}{L^{p}(\zone)}
\newcommand{\lpcone}{L^{p}(\cone)}
\newcommand{\lphone}{L^{p}(\hone)}

\newcommand{\lqrn}{L^{q}(\rn)}
\newcommand{\lqtn}{L^{q}(\tn)}
\newcommand{\lqzn}{L^{q}(\zn)}
\newcommand{\lqcn}{L^{q}(\cn)}
\newcommand{\lqhn}{L^{q}(\hn)}

\newcommand{\lo}{L^{1}}
\newcommand{\lt}{L^{2}}
\newcommand{\li}{L^{\infty}}
\newcommand{\beqn}{\begin{eqnarray*}}
\newcommand{\eeqn}{\end{eqnarray*}}
\newcommand{\pplus}{P_{Ker[\cl_+]^\perp}}

\newcommand{\co}{C^{1}}
\newcommand{\ci}{C^{\infty}}
\newcommand{\coi}{C_0^{\infty}}

\newcommand{\ca}{\mathcal A}
\newcommand{\cs}{\mathcal S}
\newcommand{\cm}{\mathcal M}
\newcommand{\cf}{\mathcal F}
\newcommand{\cb}{\mathcal B}
\newcommand{\ce}{\mathcal E}
\newcommand{\cd}{\mathcal D}
\newcommand{\cn}{\mathcal N}
\newcommand{\cz}{\mathcal Z}
\newcommand{\crr}{\mathbf R}
\newcommand{\cc}{\mathbf C}
\newcommand{\ch}{\mathcal H}
\newcommand{\cq}{\mathcal Q}
\newcommand{\cp}{\mathcal P}
\newcommand{\cx}{\mathcal X}
\newcommand{\eps}{\epsilon}

\newcommand{\pv}{\textup{p.v.}\,}
\newcommand{\loc}{\textup{loc}}
\newcommand{\intl}{\int\limits}
\newcommand{\iintl}{\iint\limits}
\newcommand{\dint}{\displaystyle\int}
\newcommand{\diint}{\displaystyle\iint}
\newcommand{\dintl}{\displaystyle\intl}
\newcommand{\diintl}{\displaystyle\iintl}
\newcommand{\liml}{\lim\limits}
\newcommand{\suml}{\sum\limits}
\newcommand{\ltwo}{L^{2}}
\newcommand{\supl}{\sup\limits}
\newcommand{\df}{\displaystyle\frac}
\newcommand{\p}{\partial}
\newcommand{\Ar}{\textup{Arg}}
\newcommand{\abssigk}{\widehat{|\si_k|}}
\newcommand{\ed}{(1-\p_x^2)^{-1}}
\newcommand{\tT}{\tilde{T}}
\newcommand{\tV}{\tilde{V}}
\newcommand{\wt}{\widetilde}
\newcommand{\Qvi}{Q_{\nu,i}}
\newcommand{\sjv}{a_{j,\nu}}
\newcommand{\sj}{a_j}
\newcommand{\pvs}{P_\nu^s}
\newcommand{\pva}{P_1^s}
\newcommand{\cjk}{c_{j,k}^{m,s}}
\newcommand{\Bjsnu}{B_{j-s,\nu}}
\newcommand{\Bjs}{B_{j-s}}
\newcommand{\Ly}{L_i^y}
\newcommand{\dd}[1]{\f{\partial}{\partial #1}}
\newcommand{\czz}{Calder\'on-Zygmund}
\newcommand{\chh}{\mathcal H}

\newcommand{\lbl}{\label}
\newcommand{\beq}{\begin{equation}}
\newcommand{\eeq}{\end{equation}}
\newcommand{\beqna}{\begin{eqnarray*}}
\newcommand{\eeqna}{\end{eqnarray*}}
\newcommand{\bp}{\begin{proof}}
\newcommand{\ep}{\end{proof}}
\newcommand{\bprop}{\begin{proposition}}
\newcommand{\eprop}{\end{proposition}}
\newcommand{\bt}{\begin{theorem}}
\newcommand{\et}{\end{theorem}}
\newcommand{\bex}{\begin{Example}}
\newcommand{\eex}{\end{Example}}
\newcommand{\bc}{\begin{corollary}}
\newcommand{\ec}{\end{corollary}}
\newcommand{\bcl}{\begin{claim}}
\newcommand{\ecl}{\end{claim}}
\newcommand{\bl}{\begin{lemma}}
\newcommand{\el}{\end{lemma}}
\newcommand{\dea}{(-\De)^\be}
\newcommand{\naa}{|\nabla|^\be}
\newcommand{\cj}{{\mathcal J}}

\title[Kinks of fractional $\phi^4$ models]
{Kinks of  fractional $\phi^4$ models: existence,  uniqueness, monotonicity, stability,  and sharp asymptotics}

\author[Atanas G. Stefanov]{\sc Atanas G. Stefanov}
\address{  
Department of Mathematics, 	University of Alabama at Birmingham,
University Hall 4049, Birmingham, AL 35294, USA}
\email{stefanov@uab.edu}

\author[P.G. Kevrekidis]{\sc P.G. Kevrekidis }
\address{Department of Mathematics and Statistics, University of Massachusetts,Amherst, MA 01003-4515, USA}
\address{Department of Physics, University of Massachusetts,Amherst, MA 01003-4515, USA}
\email{kevrekid@umass.edu}

\thanks{This material is based upon work supported by the U.S. National Science Foundation under the award \# 2204788 (A.G.S) and PHY-2110030,
PHY-2408988 and DMS-2204702 (P.G.K.). }

\subjclass[2010]{Primary  }


\date{\today}
 
\begin{abstract}
 In the present work we construct kink solutions for different (parabolic and wave) variants of the fractional $\phi^4$ model, in both the sub-Laplacian and   super-Laplacian setting. We establish existence and monotonicity results (for the sub - Laplacian case), along with sharp asymptotics which are corroborated through
 numerical computations. Importantly, in the sub-Laplacian regime, we provide the explicit and numerically verifiable spectral condition, which guarantees uniqueness for odd kinks.
 We  check numerically the relevant condition to confirm
 the uniqueness of such solutions.
 In addition, we show asymptotic stability for the stationary kinks in the parabolic setting and also, the spectral stability for  the traveling kinks in the corresponding  wave equation. 
\end{abstract}

\maketitle
\section{ Introduction}

The subject of fractional models in both dissipative
and conservative media described by nonlinear
differential equations has seen considerable growth
recently, as has been documented in numerous
reviews and books~\cite{podlubny1999fractional,cuevas,Samko,Mihalache2021}.
Relevant areas of corresponding applications extend
from optical media~\cite{malomed} to epidemic 
modeling, e.g., for measles~\cite{qureshi2020real},
from biological systems~\cite{IonescuCNSNS2017}
to economic growth models~\cite{ming2019application}
and from computer viruses~\cite{singh2018fractional}
to nonlinear wave phenomena~\cite{cuevas}.

In the present work we intend to provide a series
of results motivated by developments in fields
of dissipative (parabolic) and conservative (wave)
nonlinear partial differential equations. In that context,
while we will preserve the corresponding integer 
temporal (respectively, first or second) derivatives,
we will model spatial varying  features
through the inclusion of Riesz derivatives which 
constitute a promising framework for the corresponding
physical modeling~\cite{muslih2010riesz}. One of the many
corresponding examples has been the use of such continuum
limit representations for interacting particle systems
incorporating long-range interactions~\cite{tarasov2006continuous}.

The more concrete paradigm of interest to our study
is given by the widely used dispersive model of the
$\phi^4$ type~\cite{p4book}, as well as its dissipative
variant, namely the Allen-Cahn equation~\cite{Bartels2015}.
One of the remarkable recent developments in nonlinear optics
has been the ability of experimental groups to control
the order of dispersion, as manifested, e.g., in~\cite{BlancoRedondoNC2016,RungeNP2020}. Such
control was initially reported and accordingly leveraged
for integer orders of dispersion~\cite{TamOL2019,TamPRA2020,BandaraPRA2021,aceves}
(extending corresponding earlier theoretical
work, e.g., in~\cite{KarlssonOC1994,AkhmedievOC1994,KarpmanPLA1994}). 
However, in recent years, relevant proposals were
extended to the regime of fractional derivatives~\cite{Longhi:15}
and were subsequently implemented experimentally not only
in the linear regime~\cite{malomed}, but also more
recently in the nonlinear regime~\cite{hoang2024observationfractionalevolutionnonlinear}.
This can be characterized as a milestone development,
since it experimentally realized a Riesz derivative
corresponding to the square root of the Laplacian,
as well as additional sub-Laplacian dispersion
operators ``in the vicinity'' of a Riesz exponent 
$\alpha=1$ (see also below). The methodology utilized
therein (and private communications with this group)
reflect the potential to realize arbitrary 
fractional exponents henceforth, paving the
way for a systematic study of their implications for
wave (and dissipative) phenomena henceforth. 

\section{Mathematical Setup and Main Results}

With the above considerations in mind, we shall consider evolutionary fractional models of parabolic type, such as 
 \begin{equation}
 	\label{10} 
 	u_t + D^\alpha u+ u f(u^2-1) =0, t\geq 0, x\in\rone.
 \end{equation}
 where $\al>0$ and $D=\sqrt{-\p_{xx}}$ is the Zygmund operator, defined via the Fourier multiplier $\widehat{D g}(\xi)=2\pi |\xi|\hat{g}(\xi)$, see Section \ref{sec:2} below.

 A dispersive model of interest is given by the following fractional Klein-Gordon  equation
 \begin{equation}
 	\label{20} 
 	u_{tt} - u_{xx} + D^\alpha u+ u f(u^2-1) =0, t\geq 0, x\in\rone.
 \end{equation}
 In both examples \eqref{10}, \eqref{20}, we shall be interested in heteroclinic solutions (both traveling and stationary), which connect the equilibrium points  $\pm 1$. 
Part of our motivation, in line with earlier work
on super-Laplacian models~\cite{DeckerJPA2020,DeckerCNSNS2021,TsoliasJPA2021,TSOLIAS2023107362} is that these real field theories share 
many of the stationary solutions of the Schr{\"o}dinger-like
models, and constitute excellent entry points towards 
developing a fundamental understanding of the parabolic
and wave system features.

To make things more concrete, we will settle for the rest of our presentation on the standard ``Di Giorgi'' nonlinearity, corresponding to  $f(x)=x$. In such a case, a stationary kink $\phi$ of \eqref{10} will take the form 
 \begin{equation}
 	\label{30} 
 	D^\alpha \phi+\phi(\phi^2-1)=0, x\in \rone, 
 \end{equation}
 where we furthermore will subject it to the boundary conditions, $\lim_{x\to \pm \infty} \phi(x)=\pm1$. As an example, the case $\alpha=2$ is of course classical, the solution to \eqref{30} exists and it is unique. One can in fact write an explicit solution in the form $W(x)=\tanh\left[\f{x}{\sqrt{2}}\right]$. Note that a straightforward Fourier transform calculation yields, 
$
\widehat{\tanh\left[\f{\cdot}{\sqrt{2}}\right]}(\xi)=-\f{i \pi \sqrt{2}}{\sinh[\sqrt{2}\pi^2 \xi]}, 
$
which is definitely square integrable near $\infty$.  We also infer the asymptotic near zero, associated with the tail of the wave: 
$$
|\widehat{D^{\f{\al}{2}} W}(\xi)|=|2\pi \xi|^{\f{\al}{2}}| |\widehat{\tanh\left[\f{\cdot}{\sqrt{2}}\right]}(\xi)|\sim \f{1}{|\xi|^{1-\f{\al}{2}}}. 
$$
So,   $ D^{\f{\al}{2}}W\in L^2(\rone)$, but only for $\al>1$. 
Thus, we henceforth fix the {\it   odd function} $W$, which satisfies 
 \begin{equation}
 	\label{25} 
 	|1-W^2(x)|\leq C e^{-a|x|}, \ \ a>0; \ \ W\in \dot{H}^{\f{\al}{2}}(\rone)\cap L^\infty(\rone).
 \end{equation}

 In the case of \eqref{20}, an appropriate model can be built upon traveling kink solutions $\Phi(x-c t), |c|<1$, in the form 
 \begin{equation}
 	\label{35} 
 	-(1-c^2)\Phi''+ D^\alpha \Phi+\Phi(\Phi^2-1)=0,
 \end{equation}
again, supplemented by  $\lim_{x\to \pm \infty} \Phi(x)=\pm1$. 
 
 We now review the literature, which is vast. We will therefore not even attempt to provide a comprehensive review, but we will only touch upon results  insofar as they concern our concrete subject matter. The classical problem 
 $$
 \Delta u(x)+f(u(x))=0, \ \ x\in \Om
 $$
 was studied, among other authors, in \cite{BCN2}, where the authors investigated uniqueness and monotonicity in a bounded domain. This was later extended to fully non-linear problems. For the fractional problem, we mention the work \cite{DSV}, where the following over-determined problem was considered 
 $$
 \left\{
 \begin{array}{cc}
 D^s u=f(u(x)) & x\in\Omega \\
 u>0 & x\in \Om  \\
 u=0 & x\in \Om^c \\
 \p_n u = const & x\in\p\Om.
 \end{array}
 \right.
 $$
 where $s\in (0,2)$ and $\Om:=\{x\in\rn: x_n>\vp(x')\}$. Monotonicity of the solution was obtained, and then, under appropriate conditions on $f, \Om$, it was shown that $\Om$ must be half-space. 
 
The case of the whole space, $\rone$, was considered in a very recent result, \cite{chen}.  This is directly relevant to our considerations, as we actually use its conclusion, and it appears as Theorem 4 in \cite{chen}. We provide a representative corollary of it, as follows. 
\begin{theorem}[Wu-Chen, \cite{chen}] 
	\label{theo:chen} 
	Let $\al\in (0,2)$. Let $u\in C^{1,1}_{loc.}(\rone)\cap L^\infty(\rone)$ be a solution to 
	$$
	D^\al u (x) = f(u(x)), x\in\rone  
	$$
	where $\lim_{x\to \pm\infty} u(x)=\pm 1$ and $|u(x)|\leq 1$. Assume that $f$ is continuous on $[-1,1]$, non-increasing on $[-1, -1+\de]\cup [1-\de, 1]$ for some small $\de>0$.  Then, $u$ is strictly monotonically increasing. 
\end{theorem}

 \subsection{Main results}
 
 Our results are in several different directions. We first consider the sub-Laplacian regime $\al\in (1,2)$. We present an existence result for odd kinks, complemented with monotonicity statement and sharp asymptotics at $\pm \infty$, see Theorem \ref{theo:10} below.  Then, we present a conditional uniqueness result for odd kinks, which holds under a concrete spectral condition, see \eqref{ep:10}. Then, we move onto an asymptotic stability results for the sub-Laplacian kinks, as stationary solutions of the  parabolic problem \eqref{10}. Next, we discuss existence and spectral stability result for subsonic traveling kinks for the fractional wave equation, \eqref{20}. Finally, we provide an existence result for the super-Laplacian case $\al\in (2,4)$, see Theorem \ref{theo:40}. 
 
 \subsection{Existence of kinks: sub-Laplacian case}
 \begin{theorem}[Existence of kinks in sub-Laplacian regime]
 	\label{theo:10} 
 	
 	Let $\al\in (1,2)$. Then, the equation \eqref{30}, subject to $\lim_{x\to \pm \infty} \phi(x)=\pm 1$, 
 	 has  an {\bf odd} solution $\phi\in L^\infty(\rone) \cap 
 	\left(\cap_{s\geq \f{\al}{2}} \dot{H}^s(\rone)\right)$, in particular  $\phi\in C^\infty(\rone)$. In addition,  
 	\begin{enumerate}
 		\item   The kink $\phi$ satisfies  
 		\begin{equation}
 			\label{300} 
 			|\phi'(x)|\leq C (1+|x|)^{-1-\al}, \ \ |\phi(x)-sgn(x)|\leq C (1+|x|)^{-\al}.
 		\end{equation}
 	In fact,  $\phi$ obeys  the following sharp asymptotics for $|x|>1$, 
 		\begin{eqnarray}
 			\label{310} 
 			\phi'(x) &=&\f{2^{\al-2} \al(\al-1) \Ga\left(\f{\al-1}{2}\right)}{\sqrt{\pi} \Ga\left(\f{2-\al}{2}\right)}     |x|^{-1-\al}+O(|x|^{-3}), \\
 			\label{320} 
 			\phi(x) &=&  sgn(x) - sgn(x)  \f{2^{\al-2} (\al-1) \Ga\left(\f{\al-1}{2}\right)}{\sqrt{\pi} \Ga\left(\f{2-\al}{2}\right)}     |x|^{-\al}+O(|x|^{-2}).
 		\end{eqnarray}
 	\item $\phi'\in \cap_{s\geq 0} \dot{H}^s(\rone)$ 
 	\item The linearized operator 
 	$$
 	\cl:= D^\al+2-3(1-\phi^2)\geq 0
 	$$ 
 	and $\cl[\phi']=0$. In particular,   $\phi'>0$, $\phi$ is monotone increasing and $-1<\phi(x)<1$.  
 	
 	As a consequence, the kink $\phi$ is a local minimizer\ $\phi$ is a local minimizer for \eqref{48}, see below. 
 	\end{enumerate}
 \end{theorem}

\begin{figure}[h]  
    \centering
    \includegraphics[width=0.8\textwidth]{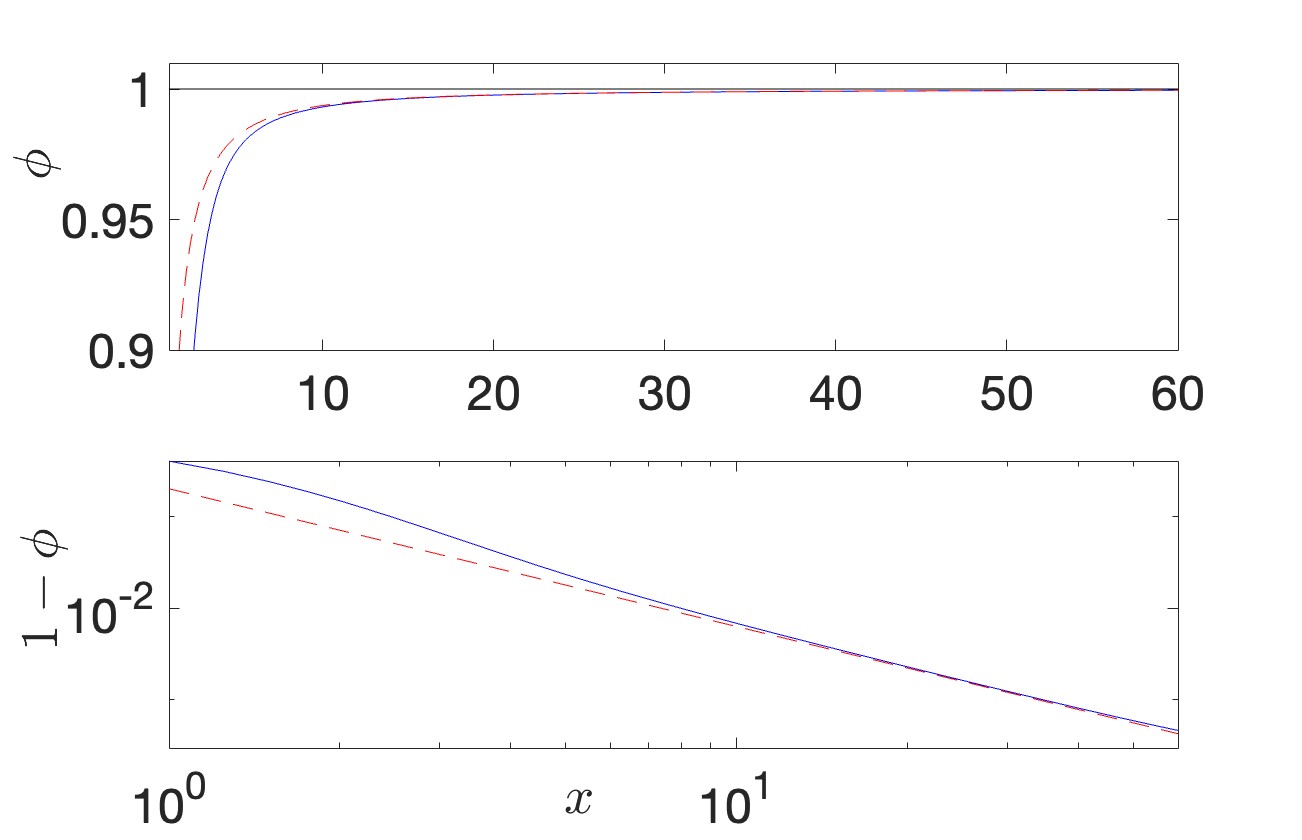}  
    \caption{In the top panel we showcase a 
    prototypical example of the decay of the 
 stationary $\phi^4$ kink to the asymptotic state
    of $\phi=1$ for the case of $\alpha=1.5$
    (the asymptotic value is denoted by the
    black line). The bottom panel shows the
    distance from the asymptotic value in a loglog
    plot. The prediction of Theorem~\ref{theo:10} 
    is shown by the dashed red line, while
    the stationary kink of the model is shown
    by the blue solid line. It can be clearly
    discerned (indeed, especially so in the
    loglog plot) that the kink tail aligns
    itself sharply with the proposed
    asymptotics of the theorem at sufficiently
    large distances.}
    \label{fig:atan_frac_f1}  
\end{figure}

 {\bf Remarks:} 
 \begin{itemize}
 	\item In \cite{chen}, a solution $u$, with the properties $\lim_{x\to \pm \infty} u(x)=\pm 1$, $-1<u(x)<1$  is proven to be monotone increasing. However, an  existence result like Theorem \ref{theo:10} seems not to have been reported previously in the literature, to the
    best of our knowledge. 
 
 	\item Related to the previous point, the monotonicity of the kink $\phi$ was established as a consequence of the variational construction (and more precisely from the positivity of the linearized operator), rather than through sophisticated moving plane methods, \cite{chen}. 
 	\item Our approach definitely fails for $\al\leq 1$, although it is conceivable that \eqref{30} has kink solutions. Indeed, we set on resolving a variational problem, which requires that $D^{\f{\al}{2}} \phi\in L^2(\rone)$.  However,  kink solutions of \eqref{30} must satisfy the asymptotics \eqref{300}, in particular $\phi'\in L^1(\rone)$. Also, it is not hard to see that 
 	$$
 	\hat{\phi}(\xi)=\f{\widehat{\phi'}(\xi)}{2\pi i \xi},
 	$$
 	in the sense of distributions, say we interpret this last as $p.v. \f{\widehat{\phi'}(\xi)}{2\pi i \xi}$.   In such a case, 
 	$$
 	\widehat{D^{\f{\al}{2}} \phi}(\xi)=(2\pi)^{\f{\al}{2}}  \f{|\xi|^{\f{\al}{2}}}{2\pi i \xi}\widehat{\phi'}(\xi)\sim \f{1}{|\xi|^{1-\f{\al}{2}}}, \ \ |\xi|\sim 0, 
 	$$
 	as $\widehat{\phi'}(\xi)$ is a continuous function, with $\widehat{\phi'}(0)=\int_{\rone} \phi'(x)dx=2$. Clearly, if  $\al\leq 1$, one has that  $\f{1}{|\xi|^{1-\f{\al}{2}}}\notin L^2(-1,1)$, which prevents this line of attack. 
 	
 	\item It is possible to consider other non-linearities in \eqref{30}, such as $f(z)=z^{2k+1}, k=1,2,\ldots$ or even polynomials $f(z)=\sum_{k=0}^n a_k z^{2k+1}$, where $a_k\geq 0$. The variational arguments presented herein are versatile enough to handle such cases, but we chose not do so, for sake of simplicity. 
 	
 	\item The remainder terms $O(|x|^{-3})$, $O(|x|^{-2})$ in \eqref{310} and \eqref{320} respectively, may be improved  to $O(|x|^{-2-\al})$ and $O(|x|^{-1-\al})$. We chose these slightly less effective bounds for the simplicity of the exposition. 
 	\item We have established   that $\cl\geq 0$, and in fact zero is a simple eigenvalue, sitting on the bottom of its spectra, $\la_0(\cl)=0$. In addition, we know by Weyl's theorem, that $\si_{a.c.}(\cl)=[2, +\infty)$. It is of course possible that some more gap eigenvalues appear (i.e. eigenvalues in $[0,2]$). This is in fact relevant for the uniqueness question, see Theorem \ref{theo:un} below, specifically the spectral condition \eqref{ep:10}.  
    \item Regarding the local minimization claim -  $\phi$  might very well be global minimizer, we just do not have a proof of this fact. For example, under the uniqueness condition, \eqref{ep:10}, it is definitely the case that $\phi$ is a global minimizer for the functional $I[u]$, see \eqref{48}. 

    \item Our numerical computations in this sub-Laplacian case
    have confirmed the relevant asymptotics, as is shown in 
    Fig.~\ref{fig:atan_frac_f1}. There, it can be seen, most notably
    in the loglog plot of the bottom panel that the tail of the 
    kink's approach to the relevant asymptotic state of $\phi=\pm 1$
    is definitively aligning with the provided asymptotic
    expression of Eq.~(\ref{theo:10}).
 	
 \end{itemize}

We next discuss the uniqueness  of the kinks for \eqref{30}. 

\subsection{(Conditional) Uniqueness of the kinks}

A natural question arises: can one come up with a condition (verifiable only in terms of $\phi$, and preferably simple to check, at least numerically) on $\phi$  which guarantees that $\phi$ is unique solution to \eqref{30}, modulo translations? The answer is positive, under a  spectral   condition as presented below, see \eqref{ep:10}.   
\begin{theorem}[Conditional Uniqueness for the kinks]
	\label{theo:un}

	Let $\al\in (1,2)$ and $\phi$ is the odd kink guaranteed by Theorem \ref{theo:10}. Assume, in addition, that the second smallest gap eigenvalue, if any, is bigger than $1$. That is   
	\begin{equation}
		\label{ep:10} 
		\la_1(\cl_\phi)> 1 \ \ \ \textup{or} \ \ \ \cl|_{\{\phi'\}^\perp} > 1
	\end{equation} 
Then,  $\phi$ is the unique odd solution to \eqref{30}. 

More precisely, assume that $u$ is an odd solution of \eqref{30}, in the sense of Theorem \ref{theo:chen}, such that 
$
1-u^2\in L^2(\rone). 
$
That is, assuming 
$$
\lim_{x\to \pm \infty} u(x)=\pm 1, \ \ 1-u^2\in L^2(\rone), \ \  |u(x)|\leq 1.
$$
and \eqref{ep:10}, it must be that   $u(x)=\phi(x)$. 
\end{theorem}
{\bf Remark:} Per Figure~\ref{fig:2} below, showcasing the second
smallest eigenvalue of the relevant operator as identified 
numerically, it becomes clear that the condition \eqref{ep:10} does indeed hold, for the numerically obtained kink $\phi$, and hence it is unique.  

\begin{figure}[h]  
\label{fig:2}
    \centering
    \includegraphics[width=0.8\textwidth]{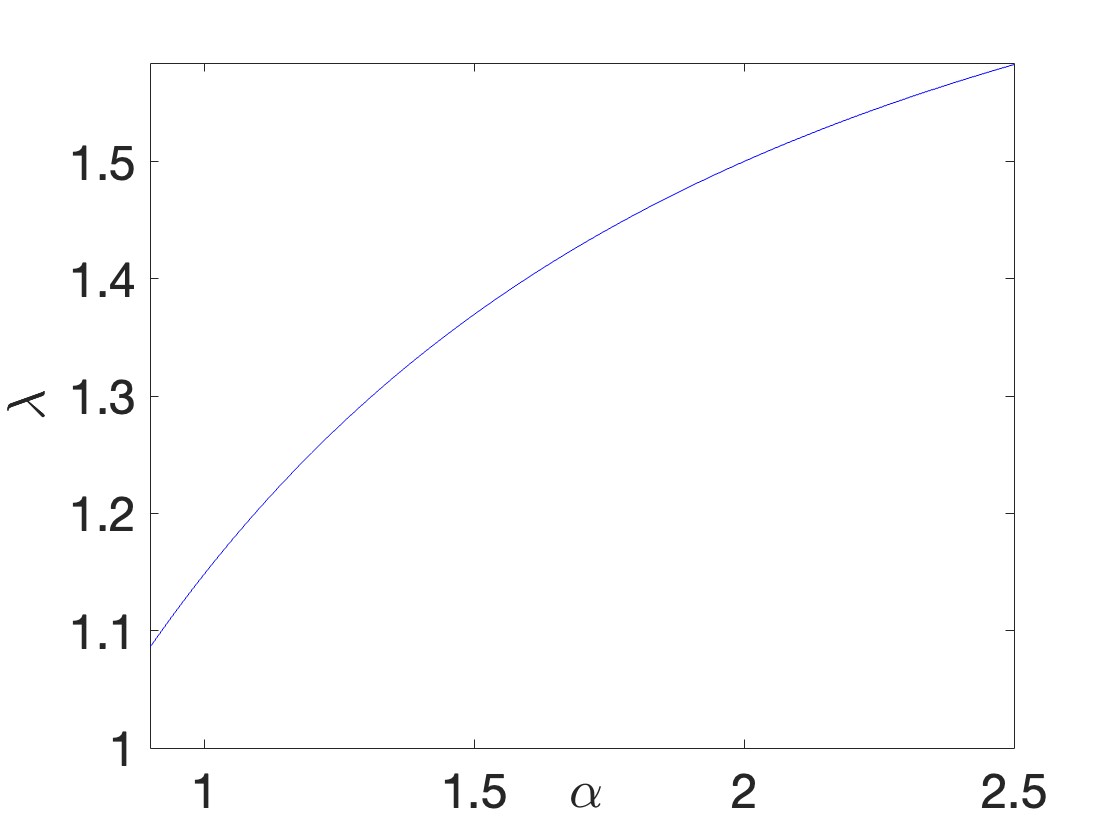}  
    \caption{The present figure effectively
    encompasses the results of our continuation
    for the kink waveform of the
    fractional wave PDE with $\alpha \in [0.9,2.5]$. What is shown is the 2nd
    eigenvalue of the (opposite of the) linearization operator,
    i.e., $\la_1(\cl_\phi)$ as a function of
    $\alpha$. It can be clearly discerned that
    thoughout the relevant interval and indeed
    for all values of $\alpha$ higher than those
    shown, this eigenvalue will satisfy the
    conditions of Theorem~\ref{theo:un},
    ensuring not only the spectral
    stability, but also the uniqueness of
    the relevant kinks. }
    \label{fig:atan_frac_f4}  
\end{figure}

Next, we discuss the kinks, in the context of the time dependent problem that they arise from, i.e., \eqref{10}. 

\subsection{Asymptotic stability with phase}

In the framework of the fractional parabolic problem \eqref{10} of which $\phi$ is a stationary solution, the positivity of $\cl$ is nothing but a spectral stability result. Our next goal is to work out a true nonlinear stability result, which will be the subject of our Theorem \ref{theo:20} below. 

In order to initiate the discussion, consider the fractional Schr\"odinger operator $\cl_\phi$. This is a self-adjoint operator, with domain $H^\al(\rone)$. According to Theorem \ref{theo:10}, $\cl\geq 0$, with a positive eigenfunction $\phi'>0$ at zero eigenvalue. As we discuss later in the proof, see Section \ref{sec:3} below, such operators enjoy the Perron-Frobenius property, \cite{FL}.  That is - the eigenvalue of the bottom of the spectrum, in this case zero, is simple  and isolated. Since the potential $3(1-\phi^2)$ has sufficient decay, \eqref{300}, 
it follows by the Weyl's theorem, that the essential (or absolute continuous) spectrum of such operator is in fact $\si_{ess.}(\cl_\phi)=[2, +\infty)$. This assures the ``spectral gap property'' for $\cl_\phi$ - namely 
\begin{equation}
	\label{340} 
	\min \{\la>0: \la\in \si(\cl_\phi)\}=:\ka>0.
\end{equation}
Based on our numerical computations $1<\ka\leq 2$.
So, we just take the value of $\ka$, 
depending on the linearized operator, whatever it may be, in our statements.

\begin{theorem}
	\label{theo:20} 
	Let $\al\in (1,2)$. Then, the odd kink solution $\phi$ described in Theorem \ref{theo:10} is asymptotically stable with phase. More precisely, there exists $\eps_0>0$ and a constant $C$, so that whenever an initial datum $u_0\in \phi+H^\al(\rone)$, is so that $\|u_0-\phi\|_{H^\al(\rone)}<\eps_0$, then the Cauchy problem for the parabolic equation \eqref{10} has a global solution. 
	
	In addition, there is a differentiable function $\si(t):[0, \infty)\to \rone$, so that $\si'(t)\in L^1(\rone_+)$, so that 
	$$
	\|u(t,\cdot)-\phi(\cdot+\si(t))\|_{H^\al(\rone)}\leq C \|u_0-\phi\|_{H^\al(\rone)} e^{-\ka t} 
	$$
	where $\ka$ is the bottom of the positive spectrum, as in \eqref{340}. In addition, $\si_\infty:=\lim_{t\to \infty} \si(t)$ exists, with an    enhanced exponential rate of convergence towards it, 
	$$
	|\si(t)-\si_\infty|\leq C  \|u_0-\phi\|_{H^\al(\rone)}^2  e^{-2\ka t} 
	$$
	If the initial perturbation is odd, then $\si(t)=0$, and we simply have 
	$$
	\|u(t)-\phi\|_{H^\al(\rone)}\leq C \|u_0-\phi\|_{H^\al(\rone)} e^{-\ka t} 
	$$
	
\end{theorem}
{\bf Remarks:} 
\begin{enumerate}
	\item The required smoothness of the perturbation above, $u_0-\phi\in H^\al$ is likely a  technical assumption, and it is almost certainly  not  optimal. This is dictated by some well-posedness theory and the (technical) requirement that $\si$ be differentiable. 
\end{enumerate}

\subsection{Kinks for the fractional wave equation}
We have already discussed the traveling kink solutions of \eqref{10}, namely they do satisfy the elliptic problem  \eqref{35}. We will present a result detailing the existence of such solutions, as well as their monotonicity and other properties, see Theorem \ref{theo:30} below. 

Let us now concentrate on the corresponding spectral  stability problem.  For a solution $\phi_c: |c|<1$ of \eqref{35}, set $u(t,x)=\phi_c(x- ct)+v(t,x-ct)$ in \eqref{20}. After ignoring all terms $O(v^2)$, we obtain the following linearized system
\begin{equation}
	\label{lin:10} 
	v_{tt}-2 c v_{tx} -(1-c^2) v_{xx}+D^\al v + 2 v- 3(1-\phi_c^2)v=0.
\end{equation}
Upon introducing the self-adjoint linearized operator $\cl:=-(1-c^2)\p_x^2+ D^\al+2-3(1-\phi^2)$, we can rewrite \eqref{lin:10} in an equivalent, first order in time system
\begin{equation}
	\label{lin:20} 
	\begin{pmatrix}
		v \\ z 
	\end{pmatrix}_t=\begin{pmatrix}
	0 & Id \\  -Id & 2c\p_x 
\end{pmatrix} \begin{pmatrix}
\cl & 0 \\ 0 & Id
\end{pmatrix} \begin{pmatrix}
v \\ z 
\end{pmatrix}
\end{equation}
Passing to the eigenvalue form $\begin{pmatrix}
	v(t,x) \\ z(t,x)
\end{pmatrix}\to e^{\la t} \begin{pmatrix}
v(x) \\ z(x) 
\end{pmatrix}$, we reduce the system to 
\begin{equation}
	\label{lin:30} 
	\begin{pmatrix}
		0 & Id \\  -Id & 2c\p_x 
	\end{pmatrix} \begin{pmatrix}
		\cl & 0 \\ 0 & Id
	\end{pmatrix} \begin{pmatrix}
		v \\ z 
	\end{pmatrix}=\la \begin{pmatrix}
	v \\ z 
\end{pmatrix}.
\end{equation}
Naturally, we say that {\it the kink is spectrally stable}, if the eigenvalue problem \eqref{lin:30} \underline{does not} have non-trivial solutions $(\la, \begin{pmatrix}
	v \\ z 
\end{pmatrix}): \Re \la>0,\begin{pmatrix}
v \\ z 
\end{pmatrix}\in D(\begin{pmatrix}
0 & Id \\  -Id & 2c\p_x 
\end{pmatrix} \begin{pmatrix}
\cl & 0 \\ 0 & Id
\end{pmatrix})$. 

We have the following existence and stability result. 
\begin{theorem}
	\label{theo:30} 
	Let $1<\al<2$. Then, \eqref{35} has an  odd solution $\phi: \lim_{x\to \pm \infty} \phi(x)=\pm 1$. In addition,  
	\begin{enumerate}
		\item   $\phi\in L^\infty(\rone) \cap 
		\left(\cap_{s\geq \f{\al}{2}} \dot{H}^s(\rone)\right)$, \ \ $\phi'\in \cap_{s\geq 0} \dot{H}^s(\rone)$.
		\item   The kink $\phi$ is monotone increasing and $-1 <\phi(x)<1$. Also, 
		\begin{equation}
			\label{3102} 
		|\phi'(x)|\leq C (1+|x|)^{-1-\al}, \ \ |\phi(x)-sgn(x)|\leq C (1+|x|)^{-\al}.
		\end{equation}
		In fact,  $\phi$ obeys  exactly the same  asymptotics as in \eqref{310}, \eqref{320}. 
		\item The linearized operator is non-negative
		$$
		\cl= -(1-c^2)\p_x^2+ D^\al+2-3(1-\phi^2)\geq 0
		$$ 
	\end{enumerate}
In addition, the kink $\phi$ is spectrally stable. More specifically, the eigenvalue problem \eqref{lin:30} does not have non-trivial solutions. 
\end{theorem}
{\bf Remarks:} The main interest in the resolution of \eqref{35} is in the fact that it  features two different (note both sub-Laplacian) dispersions, and yet the argument is versatile enough so that we are still able to resolve it, and conclude the monotonicity. 

\subsection{Existence of kinks: super-Laplacian case}
As it was discussed in the introduction, it is a valid mathematical question whether or not, the model \eqref{30} has kink solutions   for any $\al>0$. We have already addressed the case $\al\in (1,2)$, with the case $\al\in (0,1)$ left for future investigations. The case $\al=2$ is of course classical, as previously mentioned, with explicit solution $W(x)=\tanh[\f{x}{\sqrt{2}}]$, which has exponential decay for $|W^2(x)-1|+|W'(x)|$, see \eqref{25}.  

We would like to consider the cases $\al\in (2,4)$. One should keep in mind that while the solution for the case  $\al=4$ may not be explicitly known, exponential  decay as in \eqref{25} is expected.  
We note in passing that this case for the kinks has been numerically
studied, e.g., in~\cite{DeckerCNSNS2021}.
As a similar situation develops for $\al\in (4,6)$ etc, we will not discuss those cases any further. 

We have the following result. 
\begin{theorem}
	\label{theo:40} 
	Let $\al\in (2,4)$. Then, the equation \eqref{30}, subject to $\lim_{x\to \pm \infty} \phi(x)=\pm 1$, 
	has {\bf an odd} solution $\phi\in L^\infty(\rone) \cap 
	\left(\cap_{s\geq \f{\al}{2}} \dot{H}^s(\rone)\right)$ and so $\phi\in C^\infty(\rone)$. In addition,  
	\begin{enumerate}
		\item   The kink $\phi$ satisfies  
		\begin{equation}
		\label{fr:10} 
			|\phi'(x)|\leq C (1+|x|)^{-1-\al}, \ \ |\phi(x)-sgn(x)|\leq C (1+|x|)^{-\al}.
		\end{equation}
	and 
		\begin{eqnarray}
			\label{fr:20} 	
			\phi'(x) &=&-\f{2^{\al-3}\al(\al-1)(\al-2)}{\sqrt{\pi}} 
			\f{\Ga\left(\f{\al-1}{2}\right)}{\Ga\left(\f{4-\al}{2}\right)}|x|^{-1-\al}+O(|x|^{-5}), x>>1  \\
			\label{fr:23} 
			\phi(x) &=&  1+ \f{2^{\al-3}(\al-1)(\al-2)}{\sqrt{\pi}} 
			\f{\Ga\left(\f{\al-1}{2}\right)}{\Ga\left(\f{4-\al}{2}\right)}|x|^{-\al}+O(|x|^{-4}), x>>1
		\end{eqnarray}
		\item $\phi'\in \cap_{s\geq 0} \dot{H}^s(\rone)$ 
		\item The linearized operator 
		$$
		\cl:= D^\al+2-3(1-\phi^2)\geq 0
		$$ 
		and $\cl[\phi']=0$.  
	\end{enumerate}
\end{theorem}
{\bf Remarks:}

\begin{enumerate}
	\item The function $\phi'$, despite being an eigenfunction for the smallest  eigenvalue of $\cl$ is no longer assured to be positive. This is due to the fact that  $\cl$ does not have the Perron-Frobenius property, since $\al>2$.    In fact, the asymptotic formula \eqref{fr:20} confirms that, since it shows that $\phi'(x)<0$ {\bf for all large enough  $x$.  }
	\item  Related to the previous point, {\it $\phi$  no longer satisfies  $-1<\phi(x)<1$.} 
	In fact, according to the asymptotic \eqref{fr:23}, $\phi(x)>1$ for all large enough $x$. That is, $\phi(x)$ approaches $1$ from above! 
	\item The asymptotic \eqref{fr:20} is in fact equivalent to \eqref{310}, through the continuation formula for the $\Gamma$ function.
    \item Once again in this case, we test the asymptotics  against the numerical findings in 
    Fig.~\ref{fig:atan_frac_f2}. It is indeed found that the results
    of Theorem~\ref{theo:40}  provide an excellent asymptotic
    description, especially within the relevant loglog plot, of the
    tail of the kink, appropriately capturing, as is evident also
    from the top panel its approach to $\phi=1$ from above.
\end{enumerate}
 

\begin{figure}[h]  
    \centering
    \includegraphics[width=0.8\textwidth]{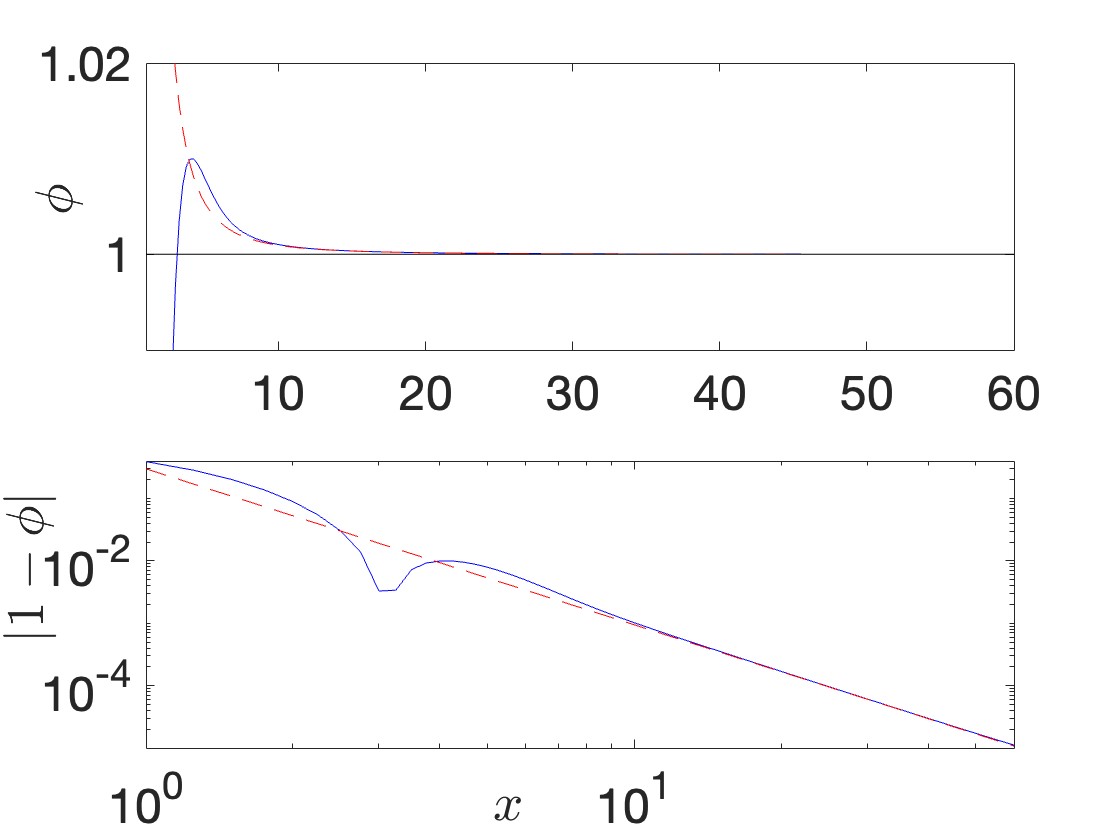}  
    \caption{Similarly to Fig.~\ref{fig:atan_frac_f1},
    in the top panel we showcase a 
    prototypical example of the decay of the 
 stationary $\phi^4$ kink to the asymptotic state
    of $\phi=1$ for the case of $\alpha=2.5$
    (the asymptotic value again denoted by the
    black line). The bottom panel shows the
    distance from the asymptotic value in a loglog
    plot. The prediction of Theorem~\ref{theo:40} 
    is shown by the dashed red line, while
    the stationary kink of the model is shown
    by the blue solid line. 
    Both the top and bottom panels clearly 
    illustrate the (single, in this case)
    crossing of $\phi=1$. In the loglog panel
    it appears as a ``dip''.
    Here, too, again both the linear and
    especially the loglog plot 
    exemplify the sharp alignment of the
    kink tail with the proposed
    asymptotics of the theorem at sufficiently
    large distances.}
    \label{fig:atan_frac_f2}  
\end{figure}

 \section{Preliminaries}
 \label{sec:2} 
 We start with some notation. The  $L^p$ spaces will be used in the standard fashion, $\|f\|_p=\left(\int_{\rone} |f(x)|^p\right)^{1/p}.$  The Fourier transform and its inverse in the form 
 $$
 \hat{f}(\xi)=\int_{-\infty}^{+\infty} f(x) e^{-2\pi i x\xi} dx,\ \ f(x) =\int_{-\infty}^{+\infty}  \hat{f}(\xi) e^{2\pi i x\xi} d\xi.
 $$
 for say $L^1(\rone)$ functions. 
Through the standard properties of the Fourier transform 
$\widehat{(-\p_{xx} f)}(\xi)=4\pi^2 \xi^2 \hat{f}(\xi)$, and the Zygmund operator is defined via 
$$
D =\sqrt{-\p_{xx}}, \ \ \widehat{D f}(\xi)=2\pi |\xi| \hat{f}(\xi).
$$
We will also routinely make use of the notation $D^s$ for an operator with a multiplier $(2\pi|\xi|)^s$. 

Alternatively and equivalently, one may introduce $D^s, s\in (0,1)$ as an operator acting as follows 
\begin{equation}
	\label{ds:10} 
	D^s u(x)=C_s \int_{\rone} \f{u(x)-u(y)}{|x-y|^{1+s}} dy.
\end{equation}
Note that the expression \eqref{ds:10} makes sense, if one only assumes that $u\in C^{1,1}_{loc.}(\rone)$ and \\ $\int_{\rone} \f{u(x)}{(1+|x|)^{1+s}}dx<\infty$, \cite{chen}. In particular, it is well-defined for $u$ bounded and $u\in C^{1,1}_{loc.}$. 

\subsection{Homogeneous tempered distributions}
Recall that  the Gamma function 
$$
\Ga(z)=\int_0^\infty t^{z-1} e^{-t} dt,
$$
is initially defined for all $\Re z>0$, but it is extendable to a meromorphic function on ${\mathbb C}$, with simple poles at $\{-k, k=0,1,2, \ldots\}$,  through the relation $\Ga(z)=\f{\Ga(z+1)}{z}$. 

For $z: -1<\Re z<0$, consider the tempered distribution (which is a locally integrable function),  
$$
u_z= \f{\pi^{\f{z+1}{2}}}{\Ga\left(\f{z+1}{2}\right)} |x|^z.
$$
One could eventually extend this definition to all $z\in {\mathbb C}$, see for example Section 2.4.3, \cite{Graf}. Specifically, for our purposes, say $-N-2<\Re z$, this can be accomplished through the formula (see  $(2.4.7)$, p. 128 in \cite{Graf})
\begin{eqnarray*}
  & &   \dpr{u_z}{f} = \int\limits_{|x|>1} f(x) u_z(x) dx+\int\limits_{|x|<1} u_z(x)\left(f(x)-\sum_{l=0}^N \f{f^{(l)}(0)}{l!} x^l\right) dx+ \sum_{l=0}^N \f{f^{(l)}(0)}{l!} \int\limits_{|x|<1} u_z(x)  x^l dx \\
& =&  \int\limits_{|x|>1} f(x) u_z(x) dx+\int\limits_{|x|<1} u_z(x)\left(f(x)-\sum_{l=0}^N \f{f^{(l)}(0)}{l!} x^l\right) dx+  \sum_{k=0}^{[N/2]} \f{f^{(2k)}(0)}{(2k)!}  \f{2\pi^{\f{z+1}{2}}}{(2k+z+1)\Ga\left(\f{z+1}{2}\right)},
\end{eqnarray*}
where we have used that for $z: z\neq -(2k+1)$
$$
\int_{|x|<1} |x|^z x^l dx=\left\{\begin{array}{cc}
     0 & l=2k+1 \\
     \frac{2}{2k+z+1} & l=2k
\end{array}    \right.
$$
 Note that the expression in the denominator allows even for $z=-(2k+1)$, due to the  analytic continuation in the denominator and  the poles of $\Ga(z)$. 

The following lemma provides the Fourier transform of this distribution, see Theorem 2.4.6 in \cite{Graf}. 
\begin{lemma}
	\label{le:126} 
	For all $z\in{\mathbb C}$, 
	\begin{equation}
		\label{126}
			\hat{u}_z=u_{-1-z}
	\end{equation}
\end{lemma}

We shall be also making use of the following standard convolution inequality - for all $\be_1, \be_2>0$, 
\begin{equation}
	\label{250} 
	\int_{-\infty}^\infty \f{1}{(1+|x-y|)^{1+\be_1}} \f{1}{(1+|y|)^{1+\be_2}} dy\leq \f{C_{\be_1, \be_2}}{(1+|x|)^{1+\min(\be_1,\be_2)}}
\end{equation}

\subsection{Littlewood-Paley partitions and various function spaces}
Introduce an even  $C^\infty_0$ function $\chi$, so that $0\leq \chi\leq 1$, 
$$
\chi(z)=\left\{
\begin{array}{cc}
	1 & |z|<1 \\
	0 & |z|>2.
\end{array}
\right.
$$
Also, $\psi(z):=\chi(z)-\chi(2z)$, so that $
\chi+\sum_{k=1}^\infty \psi(2^{-k}z)=1$  the corresponding Littlewood-Paley operators $P_k$ are defined via 
$$
\widehat{P_k f}(\xi)=\psi(2^{-k}\xi) \hat{f}(\xi).
$$
It is sometimes convenient to realize these norms by computing the norms of the Littlewood-Paley square functions 
$$
S_\ga f(x)=\left(\sum_k 2^{2k\ga} |P_k f(x)|^2\right)^{1/2}.
$$
The basic Littlewood-Paley theorem states 
$$
\|f\|_{L^p}\sim \|S_0 f\|_{L^2}. 
$$
and  as a consequence, 
$
\|D^\ga f\|_p\sim \|S_\ga f\|_p.
$
One can now define the (fractional) Sobolev spaces $W^{\ga,p}(\rone), \dot{W}^{\ga, p}(\rone), 1<p<\infty, s\geq 0$,
as the completion of the class of Schwartz functions in the norms 
$$
\|f\|_{W^{\ga, p}(\rone)}=\|f\|_p+\|D^\ga f\|_p, \ \ \ \|f\|_{\dot{W}^{\ga, p}(\rone)}= \|D^\ga f\|_p.
$$
We also make use of the standard Lipschitz spaces $C^\be(\rone), \be\in (0,1)$, defined via the norm 
$$
\|f\|_{C^\be}:=\|f\|_{L^\infty}+ \sup_{x\in\rone}\sup_{h\neq 0} \f{|f(x+h)-f(x)|}{|h|^\be}.
$$
For these spaces, there is the convenient LP characterization, (see for example Theorem 5.1.2, \cite{Graf}), 
$$
\|f\|_{C^\be}\sim \|P_{<0} f\|_{L^\infty}+\sup_{k\geq 1} 2^{k\beta} \|P_k f\|_{L^\infty},
$$
which easily implies the Sobolev embeddxing for each $0<\be<1$, $1<p<\infty$, 
\begin{equation}
	\label{sob:10} 
	\|f\|_{C^\be(\rone)}\leq C_{\be,p} \|f\|_{W^{\be+\f{1}{p},p}(\rone)}
\end{equation}
\subsection{Estimates on fractional derivatives on nonlinear functionals }
We make use of certain estimates, which involve fractional derivatives on producrs and more general functionals. Specifically, an extension of the classical Kato-Ponce (KP) estimates is as in the following proposition
\begin{proposition}[Theorem 1.2, \cite{Li}]
	\label{prop:Li} 
	For $s\in (0,1)$, $1<p<\infty$, $1< p_1, p_2, q_1, q_2\leq \infty: 
	\f{1}{p}=\f{1}{p_1}+\f{1}{p_2}=\f{1}{q_1}+\f{1}{q_2}$, we have 
	\begin{equation}
		\label{li:10} 
		\|D^s(f g)\|_{p}\leq C (\|g\|_{p_1}\|D^s f\|_{p_2}+\|f\|_{q_1} \|D^s g\|_{q_2}).
	\end{equation}
\end{proposition}
{\bf Remark:} Note that in \eqref{li:10}, one does allow for $q_2=\infty$ and/or $p_2=\infty$. The original  KP paper did not cover this case. 
\begin{proposition}(\cite{CW}, Proposition 3.1)
	\label{le:po} 
	
	Let $\ga \in (0,1)$, $F$ is $C^1(\rone)$ function and $1<q,r<\infty, 1<p\leq \infty:  \f{1}{r}=\f{1}{p}+\f{1}{q}$. 
	 Then, assuming that $u\in L^\infty(\rone)$, 
	\begin{equation}
		\label{CW:10} 
			\|D^\ga F(u)\|_r\leq C \|F'(u)\|_p \|D^\ga u\|_q
	\end{equation}
\end{proposition}
{\bf Remark:} In the original paper, \cite{CW}, it is in fact required that  
$1<p, q,r<\infty:\f{1}{r}=\f{1}{p}+\f{1}{q}$. It is however easy to extend the proof to the case $r=q \in (1, \infty),  p=\infty$, since the authors in \cite{CW} show a stronger point-wise estimate 
 $$
 S_\ga (F(u))\leq M(F’(u)) \tilde{S}_\ga (D^s u), 
 $$
where $S, \tilde{S}$ are variants of the LP square function above, and $M$ is the Hardy-Littlewood maximal operator. So, the result is even easier for this case of interest, as there is no need of H\"older's inequality.


\section{Variational construction and consequences: Proof of Theorem \ref{theo:10}}
\label{sec:3} 
 In order to set up a possible variational approach to the problem \eqref{30}, let us first address some of the  main difficulties. Specifically, the heteroclinic boundary  values do not allow for a direct construction of such functions, which belong to a good function space. So, a good roundabout way of achieving this goal is to write $\phi=W+v$, where $W$ is a {\it fixed} tanh function \eqref{25}, which possesses the correct asymptotics at $\pm\infty$ and it is a solution for exactly $\alpha=2$.  With this in mind, one may want to consider a variational problem of the type 
\begin{equation}
	\label{48} 
 I[u]:=\f{1}{2} \|D^{\al/2} u\|^2 + \f{1}{4} \int_{\rone} (1-u^2)^2 dx\to \min.
\end{equation} 
and subsequently, $u=W+v$. Taking a minimizing sequence in $v$ and attempting to take a limit   creates the usual lack of compactness issues. This may be resolvable through standard methods, such as compensated compactness, but we found it easier to approach it in a slightly different fashion. 
\subsection{Approximate variational problem}
Let $\epsilon>0$. Consider the (unconstrained variational problem)
\begin{equation}
	\label{50} 
	 I_\epsilon[u]:= \f{1}{2} \|D^{\al/2} u\|^2+ \epsilon \int_\rone (u-W)^2 \ln(e+x^2) dx + \f{1}{4} \int_{\rone} (u^2-1)^2 dx\to \min.
\end{equation}
subject to $\lim_{x\to \pm \infty} u(x)=\pm 1$. Recall that $W$ is selected to be an odd function. Our first result is that the minimizers, if any,  must be odd functions. 
\begin{lemma}
	\label{le:pol} 
	The problem \eqref{50} can be considered over the odd subspace only. In other words, if minimizers of \eqref{50} exist, they must be odd functions.
\end{lemma}
{\bf Remark:} The argument presented herein applies to all variational problems henceforth and so, we can assume that the minimizers, if they exist,  are all odd. 
\begin{proof}
	Decompose 
	$$
	u(x)=\f{u(x)+u(-x)}{2}+\f{u(x)-u(-x)}{2}=u_{even}(x)+u_{odd}(x)
	$$
	Note that $\lim_{x\to \pm \infty} u_{odd}(x)=\pm 1$, while $\lim_{x\to \pm \infty} u_{even}(x)=0$.
	Plugging this ansatz  in the functional $I_\eps$ and exploiting the parity, we obtain
	\begin{eqnarray*}
		I_\epsilon[u] &=&I_\epsilon[u_{odd}]+\f{1}{2} \|D^\f{\al}{2} u_{even}\|^2+ \epsilon \int_{\rone} u_{even}^2 \ln(e+x^2) dx + \\
		&+& \int_\rone \left[\f{u_{even}^4}{4}+\f{u_{even}^2(u_{odd}^2-1)^2}{2}+ u_{odd}^2 u_{even}^2\right]dx\geq I_\epsilon[u_{odd}].
	\end{eqnarray*}

It follows that for this  problem, constrained by $\lim_{x\to \pm \infty} u(x)=\pm 1$, $u_{odd}$  is always a strictly better option than $u$. 
\end{proof}

Back to the problem \eqref{50} - in terms of the function $v$, which does vanish at $\pm\infty$, we have the more manageable form 
\begin{eqnarray}
	\label{60} 
	J_\epsilon[v] &=& \f{1}{2} \|D^{\al/2} (v+W)\|^2+ \epsilon \int_\rone v^2 \ln(e+x^2) dx + \f{1}{4} \int_\rone ((v+W)^2-1) dx \\
	\nonumber
	&=& 
	\f{1}{2} \left(\|D^{\al/2} v\|^2+ 2 \dpr{D^\al W}{v}+ \|D^{\al/2} W\|^2\right)+  \epsilon \int_\rone v^2 \ln(e+x^2) dx + \\
	\nonumber
	&+& \f{1}{4} \int_{\rone} (W^2-1+(2 W v+v^2))^2 dx\to \min.
\end{eqnarray}
Of course, it is clear that $J_\epsilon[v]\geq 0$, while on the other hand, 
\begin{equation}
	\label{57} 
	J_\epsilon[0]= \f{1}{2} \|D^{\al/2} W\|^2+\f{1}{4} \int_{\rone} (1-W^2)^2 dx=:M_0<\infty.
\end{equation}
Introducing $\cj_\epsilon:=\inf_{v} J_\epsilon[v]>0$, we clearly have the ($\eps$ independent!) bound 
\begin{equation}
	\label{a:5} 
	\cj_\epsilon\leq J_\eps[0]=I_\eps[W]=:M_0
\end{equation}
We have the following existence result. 
\begin{proposition}
	\label{prop:10} 
	Let $\epsilon>0$. Then, the variational problem \eqref{50} has a solution $u_\epsilon$. More specifically, the equivalent variational problem \eqref{60} has a solution $v_\epsilon\in H^\alpha(\rone)$. In addition, $v$ satisfies the Euler-Lagrange equation 
	\begin{equation}
		\label{70} 
		D^\alpha (v_\epsilon+ W)+2 \epsilon v_\epsilon \ln(e+x^2)+(W+v_\epsilon) ((W+v_\epsilon)^2-1)=0.
	\end{equation}
Also, the  linearized operators $\cl_\eps$, with domain $D(\cl)=H^{\al}(\rone)$,   are non-negative for all $\eps>0$. More precisely, in operator sense, 
\begin{equation}
	\label{l110}
	\cl_\eps:= D^\al+2-3(1-u_\eps^2) + 2\eps\ln(e+x^2) \geq 0.
\end{equation}
\end{proposition}

\begin{proof}
	One can find a minimizing sequence of odd functions, $v_n: v_n\in H^{\al/2}$, so that $\lim_n J_\epsilon[v_n]=\cj_\epsilon$. Let us show now the necessary bounds on $\sup_n \|v_n\|_{H^{\al/2}}$. Indeed, let $N$ so that 
	$J_{\epsilon}[v_n]\leq 2 \cj_\epsilon$. Then, for all $n>N$, 
	\begin{equation}
		\label{64} 
			\f{1}{2} \left(\|D^{\al/2} v_n\|^2+ 2 \dpr{D^\al W}{v_n}+ \|D^{\al/2} W\|^2\right)+  \epsilon \int_\rone v_n^2dx \leq  J_{\epsilon}[v_n]\leq 2 \cj_\epsilon\leq 2 M_0.
	\end{equation}
	since $\ln(e+x^2)\geq 1$. 
	Furthermore, as 
	$
	 2 \dpr{D^\al W}{v_n}>- \f{1}{2}  \|D^{\al/2} v_n\|^2- 2 \|D^{\al/2} W\|^2,
	$
	we conclude that 
	\begin{equation}
		\label{61} 
			\f{1}{4}  \|D^{\al/2} v_n\|^2 + \epsilon \int_\rone v_n^2dx\leq 2 \cj_\epsilon+\|D^{\al/2} W\|^2\leq 4 M_0,
	\end{equation}
from \eqref{57}. 
	This yields a bound on $\sup_n \|v_n\|_{H^{\al/2}}\leq C_\epsilon$, as required. Furthermore, we have pre-compactness for the sequence $v_n$, at least in strong $L^2(\rone)$ sense, due to the bound 
	\begin{equation}
		\label{62} 
		\epsilon  \int v_n^2 \ln(e+x^2) dx \leq I_\epsilon[u_n].
	\end{equation}
	Indeed, this last bound implies that for all $M>0$ and for all  $n>N$, 
	$$
   \int_{|x|>M}  v_n^2(x)  dx \leq C M_0 \epsilon^{-1} \ln(M)^{-1}. 
	$$
	So, $\{v_n\}_n$ is pre-compact by the Riesz criteria, and hence we can take an $L^2$ strongly convergent subsequence, which without loss of generality, we assume  is the same sequence,  $\{v_n\}$.
	So, there exists  $v_\epsilon\in H^{\alpha/2}\cap L^2(\ln(e+|\cdot|^2))$, so that $\lim_n \|v_n-v_\epsilon\|=0$. We can, by taking a further subsequence, also assume that $D^{\alpha/2} v_n$ converges $L^2$ weakly to $D^{\alpha/2} v_\epsilon$,  as well as $v_n$ converges $L^2(\ln(e+|\cdot|^2))$  weakly to $v_\epsilon$. 
	
	Due to the Gagliardo-Nirenberg-Sobolev's inequality (recall $\alpha>1$) and the uniform bound $\sup_n \|v_n\|_{H^{\al/2}}\leq C_\epsilon$, we can extend this claim to all $2\leq p\leq \infty$ spaces, namely $\lim_n \|v_n-v_\epsilon\|_p=0$. So, 
	\begin{eqnarray*}
	& & 	\lim_n \left( \dpr{D^\al W}{v_n}+  \f{1}{4} \int_{\rone} (W^2-1+(2 W v_n+v_n^2))^2 dx\right) \\
		& & = \dpr{D^\al W}{v_\epsilon}+  \f{1}{4} \int_{\rone} (W^2-1+(2 W v_\epsilon+v_\epsilon^2))^2 dx
	\end{eqnarray*}
By the lower semi-continuity of the norms, with respect to weak norms, 
$$
\liminf_n \left(\f{1}{2} \|D^{\alpha/2} v_n\|^2+ \epsilon \int v_n^2 \ln(e+x^2) dx\right)\geq 
\f{1}{2} \|D^{\alpha/2} v_\epsilon\|^2+ \epsilon \int v_\epsilon^2 \ln(e+x^2) dx
$$
All in all, we obtain 
$$
\cj_\epsilon=\liminf_n J_\epsilon[v_n]\geq J_\epsilon[v_\epsilon],
$$
which of course implies that $v_\epsilon$ is a minimizer for the variational problem \eqref{60}, and subsequently $u_\epsilon:=W+v_\epsilon$ is a minimizer for \eqref{50}.

As $u_\epsilon$ is an (unconditional) minimizer  for $I(u_\epsilon)$, 
then clearly the function $g(z):=I_\epsilon(u_\epsilon+z h)$ has a minimum at zero for every fixed  test function $h\in C_0^\infty(\rone)$. It is easy to see that 
$$
I_\epsilon(u_\epsilon+z h)=I_\epsilon(u_\epsilon)+z \dpr{D^\al u_\eps+2\epsilon(u_\epsilon-W)\ln(e+x^2) +u_\epsilon(u_\epsilon^2-1)}{h}+O(z^2)
$$
It follows that in a distributional sense, $u_\eps$ satisfies 
$$
D^\al u_\eps+2\epsilon(u_\epsilon-W)\ln(e+x^2) +u_\epsilon(u_\epsilon^2-1)=0,
$$
which is of course \eqref{70}. 

The positivity of the operators $\cl_\eps$ follows from the fact that $u_\eps$ is an unconstrained minimizer for the functional $I_\eps$ in \eqref{50}. Indeed, for all test functions, $h$, we have that 
$g(\de):=I_\eps[u_\eps+\de h]$ achieves a minimum at $\de=0$. It follows that the first variation is zero, which yielded \eqref{70}, while on the other hand, the second variation is easily written 
$$
g(\de)=I_\eps[u_\eps]+ \f{\de^2}{2}  \dpr{ \cl_\eps h}{h}+o(\de^2),
$$
whence $\cl_\eps\geq 0$. 
\end{proof}
\subsection{Uniform in $\eps$ a priori estimates}
Clearly our goal is to pass to a limit $\eps\to 0+$ in the equation \eqref{70}. In order to ensure that such limit is possible, we shall need to establish a number of estimates for $v_\eps$, which are uniform in the range $0<\eps<1$. To this end, we have the following proposition.

\begin{proposition}[Uniform a priori estimates]
	\label{prop:20} 
	We have 
	\begin{eqnarray}
		\label{c:17} 
		\sup_{0<\eps<1} \|q_\eps\|_{H^{\f{\al}{2}}}<\infty, \\
		\label{c:18} 
		\sup_{0<\eps<1}(\|v_\eps\|_{L^\infty}+\|D^{\f{\al}{2}} v_\eps\|_{L^2})<\infty.
	\end{eqnarray}
\end{proposition}
\begin{proof}
	As we have established that $v_\eps$ is a minimizer, we have that $J_\eps[v_\eps]=J_\eps\leq M_0$, according to \eqref{a:5}. One immediate observation, say from \eqref{64},  is the uniform in $\eps$ bound 
	\begin{equation}
		\label{a:10} 
		\|D^{\alpha/2} v_\eps\|\leq C. 
	\end{equation}
	On the other hand, from \eqref{62} one has control of $\|v_\eps\|$ in the form $\|v_\eps\|\leq C \eps^{-\f{1}{2}}$, which  blows up as $\eps\to 0+$. As such, this bound is not very useful for our purposes. On the other hand,  $v_\eps\in L^\infty$ by Sobolev embedding, but at this point, with the estimates at hand,  we cannot guarantee that  $\|v_\eps\|_{L^\infty}$ does not blow up as $\eps\to 0$. 
	
	Instead, we develop alternative $L^p, 2<p\leq \infty$ and other substitutes, which will allow us to complete the said limiting procedure. Specifically, denote 
	$$
	q_\eps:= 2W v_\eps+v_\eps^2.
	$$
	We claim  a uniform in $\eps$,  bound on $\|q_\eps\|$. Indeed, we have by Cauchy-Schwartz
$$
		M_0\geq J_\eps[v_\eps]\geq 	\int_\rone (W^2-1+q_\eps)^2 dx\geq \|W^2-1\|^2+2\dpr{W^2-1}{q_\eps}+ \|q_\eps\|^2\geq  \f{1}{2} \|q_\eps\|^2 - \|W^2-1\|^2. 
$$
It follows that 
\begin{equation}
	\label{a:11} 
	\|q_\eps\|^2 \leq M_0+\|W^2-1\|^2=:M_1
\end{equation}
Next, we aim at obtaining a uniform bound $\sup_{0<\eps<1} \|v_\eps\|_{L^\infty}$. 

Recall the smooth cut-off function $\chi$, which is identically $1$ on $|z|<1$, and vanishes on $|z|>2$. Let 
$$
v_\eps=v_\eps\chi(2W+v_\eps)+ v_\eps(1-\chi(2W+v_\eps))=:v_{\eps, 1}+v_{\eps,2}.
$$
Clearly, on the support of $v_{\eps,1}$, we have $|v_\eps(x)|\leq 2+2 |W(x)|\leq 4$, whence $\|v_{\eps,1}\|_{L^\infty}\leq 4$. 

Regarding $v_{\eps,2}$, we have that on its support $|2W(x)+v_\eps(x)|\geq 1$, whence 
\begin{equation}
	\label{a:18} 
	\int_\rone v_{\eps,2}^2(x) dx \leq \int_\rone v_{\eps}^2(x)(2W+v_\eps(x))^2 dx=\|q_\eps\|^2.
\end{equation}
 This, together with \eqref{a:11}  yields a uniform $L^2$ bound for $v_{\eps,2}$. However, 
 \begin{eqnarray*}
 	\|D^{\al/2} v_{\eps,2}\| &\leq &  \|D^{\al/2} v_{\eps}\|+\|D^{\al/2}\chi(2W+v_\eps)\|\leq C \|D^{\al/2} v_{\eps}\|+ C \|D^{\al/2}[2W+v_\eps]\|\leq \\
 	&\leq & C_1 \|D^{\al/2} v_{\eps}\|+C\|D^{\al/2} W\|.
 \end{eqnarray*}
where we have used Proposition  \ref{le:po} to control 
$$
\|D^{\al/2}\chi(2W+v_\eps)\|\leq C \|\chi'(v_\eps)\|_{L^\infty} (\|D^{\al/2} W\|+\|D^{\al/2} v_\eps)\|.
$$
 All in all, we obtain uniform bounds on $\|v_{\eps,2}\|_{H^{\al/2}}$, which controls $\|v_{\eps,2}\|_{L^\infty}$. This together with the bound $\|v_{\eps,1}\|_{L^\infty}\leq 4$, provides a uniform in $\eps$ bound, which we record 
\begin{equation}
	\label{80} 
\sup_{0<\eps<1}	\|v_\eps\|_{L^\infty}\leq C. 
\end{equation}
We are then ready for uniform estimates for $q_\eps$. Indeed, from Proposition \ref{prop:Li}, specifically the estimate \eqref{li:10}, we have 
\begin{equation}
	\label{150} 
	\|D^{\f{\al}{2}} q_\eps\|\leq C (\|D^{\f{\al}{2}} v_\eps\| \|W\|_\infty + \|D^{\f{\al}{2}} W\| \|v_\eps\|_\infty+ \|D^{\f{\al}{2}} v_\eps\| \|v_\eps\|_\infty)\leq M_3,
\end{equation}
where the last uniform estimate follows from \eqref{a:10}, \eqref{80} and \eqref{25}. 
\end{proof}
We are now ready for the limit $\eps\to 0+$. 
\begin{proposition}
	\label{prop:40} 
	
	Let $\al\in (1,2)$. Then, 
	there is a distributional solution \\ $u\in L^\infty(\rone)\cap \dot{H}^{\f{\al}{2}}(\rone)\cap_{l=1}^\infty \dot{H}^l(\rone)$ of \eqref{30}. In particular, $u\in C^\infty(\rone)$. Moreover, 
	\begin{enumerate}
		\item  the  linearized operator  $\cl=D^\al+2-3(1-u^2)  \geq 0$ is non-negative, $\cl[u']=0$. 
		\item $\lim_{x\to \pm \infty} u(x)=\pm 1$ and $u$ is monotonically increasing. \\ In particular, $-1<u(x)<1$ for all $x\in \rone$. 
	\end{enumerate}  
\end{proposition}
\begin{proof}
	We already have control of $\sup_{\eps: 0<\eps<1} \|v_\eps\|_{H^{\f{\al}{2}}}$ from \eqref{c:17}. Note that by the Sobolev embedding \eqref{sob:10} and the bounds in  \eqref{c:18}, we have the following uniform bound in the Lipschitz space 
	$$
	\sup_{\eps: 0<\eps<1} \|v_\eps\|_{C^{\f{\al-1}{2}}(\rone)}<\infty
	$$
	By the Arzela-Ascoli theorem, it follows that the family $\{u_\eps\}_{0<\eps<1}$ is pre-compact in any $C[-N,N]$, and so one can select a uniformly convergent on the compact subsets of $\rone$ subsequence \\ $u_{\eps_n} \rightrightarrows u$, so that $u\in C^{\f{\al-1}{2}}(\rone)$. We test these convergences in the Euler-Lagrange equation \eqref{70}. We choose a test function $\Psi\in C^\infty_0(\rone)$, and we obtain 
	\begin{equation}
		\label{c:20} 
			\dpr{D^\alpha u_\eps}{\Psi} +2 \epsilon \dpr{v_\epsilon}{\Psi  \ln(e+x^2)} +\dpr{u_\epsilon (u_\epsilon^2-1)}{\Psi}=0.
	\end{equation}
	The uniform convergence $u_{\eps_n} \rightrightarrows u$ and the uniform bound on $\|u_\eps\|_\infty$ guarantees by the Lebesgue dominated convegrence theorem that 
	$$
	\lim_{\eps\to 0+} \dpr{u_\epsilon (u_\epsilon^2-1)}{\Psi}= \dpr{u (u^2-1)}{\Psi}.
	$$
	As $\Psi_1:=\Psi  \ln(e+x^2)$ is also a Schwartz function, we have that 
	$$
	|\dpr{v_\epsilon}{\Psi  \ln(e+x^2)} |\leq \|v_\eps\|_\infty \|\Psi_1\|_{L^1}\leq M_5, 
	$$
	whence $\lim_{\eps\to 0+}\eps \dpr{v_\epsilon}{\Psi  \ln(e+x^2)}=0$. Finally, acoording to Lemma \ref{app:10}, we have that $D^\al \Psi\in L^1(\rone)$, whence by the Lebesgue dominated convergence theorem, we conclude 
	$$
	 \lim_{\eps\to 0+} \dpr{D^\al u_\eps}{\Psi}=\lim_{\eps\to 0+} 
	\dpr{u_\eps}{D^\al  \Psi}=\dpr{u}{D^\al  \Psi}.
	$$
	In all, we recover that $u$ satisfies the equation 
	\begin{equation}
		\label{140} 
		D^\al u+ u(u^2-1)=0
	\end{equation}
	in a distributional sense. The positivity of the linearized operator is now a direct consequence. Indeed, take again a test function $\Psi\in C^\infty_0(\rone)$. We have 
	$$
	\dpr{\cl_\eps \Psi}{\Psi}= \|D^{\f{\al}{2}}\Psi\|^2+2 \|\Psi^2\|^2-3 \int_{-\infty}^\infty  (1-u_\eps^2) \Psi^2 +2\eps \int_{-\infty}^\infty \Psi^2(x) \ln(e+x^2) dx \geq 0,
	$$
	by the positivity of $\cl_\eps$. On the other hand, 
	since $u_{\eps_n}\rightrightarrows u$ on the compact subsets, 
	we clearly have by Lebesgue dominated convergence theorem
	$$
	\lim_{n\to \infty} \left[-3 \int_{-\infty}^\infty  (1-u_{\eps_n} ^2) \Psi^2 +2\eps_n \int_{-\infty}^\infty \Psi^2(x) \ln(e+x^2) dx\right] = -3 \int_{-\infty}^\infty  (1-u^2) \Psi^2. 
	$$
	In all, 
	$$
	\dpr{\cl \Psi}{\Psi} = \lim_{n\to \infty} 	\dpr{\cl_{\eps_n} \Psi}{\Psi} \geq 0.
	$$
	
	We now bootstrap the bounds on $u$ to  $\cap_{l=1}^\infty  \dot{H}^l(\rone)$ bounds. Indeed, starting with \eqref{140}, we have by using the bounds \eqref{25}
	$$
	\|D^\al u\|=\|u(u^2-1)\|\leq \|u\|_{L^\infty} \|W^2-1+2 (W^2-1) q+q^2\|\leq C(\|q\|_\infty+\|q\|_{L^4}^2)\leq C (1+\|q\|_{H^{\f{\al}{2}}}^2).
	$$
	where in the last step, we have used Sobolev embedding $H^{\f{\al}{2}}(\rone)\hookrightarrow  L^4(\rone)$. This improves the {\it a priori} bounds to $u\in \dot{H}^\al$. It follows, much as in the proof of \eqref{150} that 
	$$
	\|D^\al q\|=\|D^{\al-1} q'\| \leq C(1+\|v\|_{\dot{H}^\al}^2).
	$$
One can then apply the same approach to estimate $D^{2\al}$ from the relation \eqref{140} $\|D^{\alpha+1} u\|=\|(u(u^2-1))'\|$ and use the estimates obtained for $\|u'\|\leq \max(\|D^{\alpha} u\|, \|D^{\f{\alpha}{2}} u\|)$ and so on. In this fashion, we obtain bounds for $\|D^{l +\alpha} u\|\leq C_l, l=0,1,  \ldots$. Note however that with these arguments, we can never control homogeneous Sobolev norm below the {\it a priori} norm, that is $\|D^\be u\|, \be<\f{\al}{2}$. 

Having such a smooth solution $u$ allows us to take a spatial derivative in \eqref{140}, which implies 
$$
\cl u'=D^\al u'+2u'-3(1-u^2) u'=0.
$$
In other words, $u'\in D(\cl)$ is an eigenfunction for the linearized operator $\cl$ corresponding to the zero eigenvalue. Recall that $\cl\geq 0$, which makes zero the lowest eigenvalue for $\cl$. 	Note however,  that the operator $\cl$ satisfies the Perron-Frobenius property, \cite{FL}, whence the eigenfuncton corresponding to the  lowest eigenvalue does not change sign. It follows that $u'\geq 0$ or $u'\leq 0$. Either way, $u$ is monotone. 
	
	We now address the behavior at $\pm \infty$ of $u$. Note that the uniform convergence $u_{\eps_n} \rightrightarrows u$  over the compact 
	subsets of $\rone$ guarantees that $q_{\eps_n}\rightrightarrows q=2W v+v^2$ over the compact 
	subsets of $\rone$. In addition, from \eqref{c:17}, we have that $q\in H^{\f{\al}{2}}\subset C^{\f{\al-1}{2}}$. Thus, and by the fact that $\sup_{0<\eps<1}\|q_\eps\|_{L^2}<\infty$, it follows that $\|q\|_{L^2}<\infty$ and so 
	$\lim_{|x|\to \infty} q(x)=0$. Since $q=2W v+v^2=v(2W+v)$ this leads to two possible conclusions at each infinity. Say at $+\infty$, this means that either $\lim_{x\to +\infty} v(x)=0$ or $\lim_{x\to +\infty} v(x)=-2$. Similarly, at $-\infty$,  either $\lim_{x\to -\infty} v(x)=0$ or $\lim_{x\to -\infty} v(x)=2$. 
	
	In terms of $u$, $\lim_{x\to +\infty} u(x)=1$ or $\lim_{x\to +\infty} u(x)=-1$. Similarly,   either $\lim_{x\to -\infty} u(x)=-1$ or $\lim_{x\to -\infty} u(x)=1$. Note that by the monotonicity of $u$ established above, we have that only two of the four options above are viable, namely 
	$$
	\lim_{x\to -\infty} u(x)=-1, \lim_{x\to +\infty} u(x)=1,
	$$
	or 
	$$
	\lim_{x\to -\infty} u(x)=1, \lim_{x\to +\infty} u(x)=-1,
	$$
	Either way, we have constructed a solution with $\lim_{x\to -\infty} u(x)=-1, \lim_{x\to +\infty} u(x)=1$, since in the other possible scenario, we just perform the transformation $u\to -u$, which still yields a solution by the symmetry of the equation. 
\end{proof}

\subsection{The kink $u$ minimizes \eqref{48}}
 Even though the solution $u$ was produced as a limit of minimizers of the approximate problems \eqref{50}, we can in fact show that it does minimize $I[u]$. 
 \begin{proposition}
     \label{prop:97}
     The kink $u$ produced in Proposition \ref{prop:40} is a local minimizer of \eqref{48}. 
 \end{proposition}
 {\bf Remark:} Note that if $u$ is a unique solution to \eqref{30}, then it is in fact a unique global minimizer as well. Indeed, any other local minimizer will be an odd function, which is a solution to \eqref{30} as well. By the uniqueness, this should not be possible. So, $u$ will be a unique local minimizer, and hence a global one. 
 \begin{proof}
     The kink satisfies  $u\in \dot{H}^{\f{\al}{2}}(\rone)$ and $1-u^2\in L^2(\rone)$, so $I[u]$ makes sense.  Recall also that the minimization problem \eqref{48} may be taken in the space of odd functions only. Let $h$ be an odd test function and $\delta>0$. We have, by the Euler-Lagrange equation for $I[u]$,  which is \eqref{30}, 
     $$
     I[u+\de h]=I[u]+\f{\de^2}{2}\dpr{\cl h}{h}+O(\de^3)
     $$
     As $\cl\geq 0$, with a one-dimensional kernel at zero, spanned by the  even eigenfunction $\phi'$, we have that $\dpr{\cl h}{h}>0$, whenever $h\neq 0$.  Thus, $u$ is a strict local minimizer. 
 \end{proof}

\section{(Conditional) Uniqueness for the sub-Laplacian kinks: Proof of Theorem \ref{theo:un}}
Let $u$ be  an odd solution to \eqref{30}, with $\al\in (1,2)$, in the sense of Theorem \ref{theo:chen}. Then, as the non-linearity $f(u)=u-u^3$  satisfies the assumptions ($f'(u)=1-3u^2<0$, for $u$ close to $\pm 1$), it follows that $u$ is monotonically increasing.  
 Thus,  $u(0)=0$, and  $u(x)<0, x<0$, while $u(x)>0, x>0$. Note that  as a consequence, $\sup_{x\in\rone}|u(x)-\phi(x)|\leq 1$. 

Due to the assumption that 
$1-u^2\in L^2(\rone)$, one can iterate the smoothness and decay conditions to 
$u\in \cap_{s\geq \f{\al}{2}} \dot{H}^s(\rone), u'\in \cap_{s\geq 0} H^s(\rone)$, so in particular $u\in C^\infty(\rone)$. Note that by differentiating the equation \eqref{30}, we obtain 
$$
D^\al[u']+(3u^2-1)u'=0,
$$
whence $\cl_u:=D^\al+2-3(1-u^2)$ has an eigenvalue at zero, with an eigenfunction $u'>0$. It follows, by the Perron-Frobenius property, that $\cl_u\geq 0$. 

Consider now 
$$
D^\al \phi+\phi(\phi^2-1)=0=D^\al u+u(u^2-1)
$$
Subtracting off the two yields 
$$
D^\al(u-\phi)+(u-\phi)(u^2+u \phi+\phi^2-1)=0.
$$
Thus, the fractional Schr\"odinger operator 
$$
\cl_{u, \phi}:=D^\al+2-3(1-\f{u^2+u \phi +\phi^2}{3}),
$$
 is introduced, with $L^2$ decaying potentials at $\pm\infty$ and  $\cl_{u,\phi}(u-\phi)=0$. Note that $u-\phi\in L^2(\rone)$, by the assumptions on $u$ and the properties of $\phi$. So, $z:=u-\phi$ is an  eigenfunction,  corresponding to zero eigenvalue, for $\cl_{u, \phi}$. Also, $z$ is an odd function, and as such $z\perp \phi'$. 
 Finally, 
$$
\cl_{u,\phi}=\f{1}{2}\left(\cl_u+\cl_\phi\right)-\f{(u-\phi)^2}{2}\geq \f{1}{2}\left(\cl_\phi-(u-\phi)^2\right)\geq \f{1}{2}\left(\cl_\phi-1\right).
$$
since $\cl_u\geq 0$, $\sup_{x\in\rone} |u(x)-\phi(x)|\leq 1$. Finally, restricting to the subspace ${\{\phi'\}^\perp}$, we have 
$$
P_{\{\phi'\}^\perp}\cl_{u,\phi}P_{\{\phi'\}^\perp}  \geq \f{1}{2}\left(P_{\{\phi'\}^\perp}\cl_\phi P_{\{\phi'\}^\perp}-1\right)>0, 
$$
by  the assumption $P_{\{\phi'\}^\perp} \cl_\phi P_{\{\phi'\}^\perp}\geq 1$. But then, $z=u-\phi$ must be trivial, as it would otherwise be an actual eigenfunction, corresponding to the zero eigenvalue of the strictly positive operator  $P_{\{\phi'\}^\perp} \cl_{u, \phi}P_{\{\phi'\}^\perp} $.

\section{Decay rates and asymptotics for the kinks}
\label{sec:4} 
We now take on the decay rates for $u$. We need to first prepare with a Lemma. 
\begin{lemma}
	\label{g:10} 
	Let $\al\in (1,2)$. Then, the resolvent operator $(D^\al+2)^{-1}$  is represented as 
	$$
	(D^\al+2)^{-1} F(x)=\int_{-\infty}^\infty K_\al(x-y) F(y) dy, 
	$$
	with a positive kernel $K_\al$, which has the decay rate 
	\begin{equation}
		\label{g:25} 
		0<K_\al(x)\leq \f{C_\al}{(1+|x|)^{1+\al}}.
	\end{equation}
	Moreover, we have the precise asymptotic formula 
	\begin{equation}
		\label{g:20} 
		K_\al(x)= \f{2^{\al-4} \al(\al-1) \Ga\left(\f{\al-1}{2}\right)}{\sqrt{\pi} \Ga\left(\f{2-\al}{2}\right)}     |x|^{-1-\al}+O(|x|^{-3}), \ \  |x|>1
	\end{equation}
For large $x$, we have the estimate 
\begin{equation}
	\label{g:45} 
	|K_\al'(x)|\leq C |x|^{-3}, |x|>>1.
\end{equation}
\end{lemma}
{\bf Remark:} The error terms in both \eqref{g:25}and \eqref{g:45} may be improved to the more natural rate $O(|x|^{-2-\al})$. 
In the interest of simplifying the presentation, we provide these simpler versions. 
We postpone the somewhat technical proof of Lemma \ref{g:10} to  the Appendix.
\begin{proposition}
	\label{prop:50}  Let $\al\in (1,2)$. For the solution $u$ constructed in Proposition \ref{prop:40},  there is the sharp decay rate 
	\begin{equation}
		\label{b:18} 
			|u(x)-sgn(x)|\leq C (1+|x|)^{-\al}.
	\end{equation}
	In fact, we have the following bounds on $u'$, 
	\begin{equation}
		\label{b:21} 
	 0<u'(x)< \f{C}{(1+|x|)^{1+\al}}
	\end{equation}
Moreover,  the following asymptotic formulas hold true 
\begin{eqnarray}
	\label{260} 
	u'(x) &=&\f{2^{\al-2} \al(\al-1) \Ga\left(\f{\al-1}{2}\right)}{\sqrt{\pi} \Ga\left(\f{2-\al}{2}\right)}     |x|^{-1-\al}+O(|x|^{-3}), x>>1\\
	\label{270} 
	u(x) &=&  sgn(x) - sgn(x)  \f{2^{\al-2} (\al-1) \Ga\left(\f{\al-1}{2}\right)}{\sqrt{\pi} \Ga\left(\f{2-\al}{2}\right)}     |x|^{-\al}+O(|x|^{-2})
\end{eqnarray}
\end{proposition}
\begin{proof}
	We first show the bounds on $u'$. We have already shown that $\|u'\|_{L^\infty}<\infty$. Going forward, we assume that $x>>1$, the case of a large negative $x$ is handled in a symmetric way. 
	
Observe that  by construction, we have that  $\lim_{x\to +\infty} u(x)=1$. Thus, for every $\delta>0$, there is $k_0=k_0(\de)$, so that $0<1-u^2(x)<\de$, as long as $x>2^{k_0}$. Fix  $\de>0$ for the moment, to be determined momentarily. 

Observe that since $\cl[u']=0$, we may rewrite this as an integral equation (noting that  all the quantities inside the integral, including $K_\al$ are positive). We obtain 
\begin{eqnarray*}
	0<u'(x) &=& 3(D^\al+2)^{-1}[(1-u^2)u']=3\int_\rone K_\al(x-y)(1-u^2(y))u'(y) dy\leq \\
	&\leq & C_\al \int_\rone \f{1}{(1+|x-y|)^{1+\al}}(1-u^2(y))u'(y) dy
\end{eqnarray*}
Since $\lim_{x\to \infty} u'(x)=0$ as well, we may define $a_k:=\max_{x\geq 2^k} u'(x)=u'(x_k)$ for some $x_k\geq 2^k$. Take $k>k_0$ and evaluate at $x_k$. We obtain 
	\begin{eqnarray*}
		a_k &\leq & C_\al \int_\rone \f{1}{(1+|x_k-y|)^{1+\al}}(1-u^2(y))u'(y) dy\leq 
		C_\al \de \int_\rone \f{1}{(1+|x_k-y|)^{1+\al}} u'(y) dy \\
		&\leq & C_\al \de \left[\int_{-\infty}^{\f{x_k}{2}}  \f{1}{(1+|x_k-y|)^{1+\al}} u'(y) dy+ \int_{\f{x_k}{2}}^\infty \f{1}{(1+|x_k-y|)^{1+\al}} u'(y) dy\right].
	\end{eqnarray*}
	We estimate the first integral by Cauchy-Schwartz 
	$$
	\int_{-\infty}^{\f{x_k}{2}}  \f{1}{(1+|x_k-y|)^{1+\al}} u'(y) dy\leq \left(	\int_{-\infty}^{\f{x_k}{2}}  \f{1}{(1+|x_k-y|)^{2(1+\al)}}\right)^{\f{1}{2}} \|u'\|\leq C_\al x_k^{-\al-\f{1}{2}},
	$$
	while 
	the second is bounded 
	$$
	 \int_{\f{x_k}{2}}^\infty \f{1}{(1+|x_k-y|)^{1+\al}} u'(y) dy\leq a_{k-1}  \int_{\f{x_k}{2}}^\infty \f{1}{(1+|x_k-y|)^{1+\al}} dy\leq C_\al a_{k-1}.
	$$
	Altogether, we have 
	\begin{equation}
		\label{200} 
		a_k \leq C_\al\de a_{k-1} + C_\al 2^{-k(\al+\f{1}{2})}. 
	\end{equation}
	Letting $b_k:=C_\al^{-1} a_k$, we have 
	\begin{equation}
		\label{210} 
		b_k \leq C_\al\de b_{k-1} +  2^{-k(\al+\f{1}{2})}. 
	\end{equation}
	At this point, we select $\de=\de(\al): C_\al\de<\f{1}{100}$, and then the corresponding $k_0=k_0(\al)$. With this choices, and for  $k>k_0$, the estimate \eqref{200}can be iterated as follows 
\begin{eqnarray*}
	b_k &\leq &  \f{1}{100}b_{k-1}+2^{-k(\al+\f{1}{2})}\leq \f{1}{100}\left( \f{1}{100} b_{k-2}+2^{-(k-1)(\al+\f{1}{2})}\right)+ 2^{-k(\al+\f{1}{2})}\leq \ldots \\
	&\leq & \f{1}{100^{k-k_0}} b_{k_0}+2^{-k(\al+\f{1}{2})} \sum_{l=0}^{k-k_0} \left(\f{2^{(\al+\f{1}{2})}}{100}\right)^l\leq 
	C_\al 2^{-k(\al+\f{1}{2})}. 
\end{eqnarray*}
	It follows that $\max_{x>2^k} u'(x)\leq C 2^{-k(\al+\f{1}{2})}$, which yield the intermediate bound

	\begin{equation}
		\label{220}
	u'(x)\leq \f{C}{(1+|x|)^{\al+\f{1}{2}}}.
\end{equation}
	Here, note that the bound just established is weaker than \eqref{b:21}. We can however bootstrap it as follows. First off, 
	one can write 
		\begin{equation}
		\label{230}
	1-u(x)=\int_x^\infty u'(y) dy\leq C \int_x^\infty \f{1}{(1+|y|)^{\al+\f{1}{2}}}dy \leq \f{C}{(1+|x|)^{\al-\f{1}{2}}}. 
\end{equation}
	Thus, going to the integral equation, we can enter the new improved estimates, for $u'$ in \eqref{220},  and for $1-u$ in \eqref{230}, to obtain 
	$$
	u'(x)\leq C \int_{\rone} \f{1}{(1+|x-y|)^{1+\al}}  \f{1}{(1+|y|)^{2\al}}dy\leq \f{C}{(1+|x|)^{1+\al}},
	$$
	where in the last step, we used \eqref{250}. This is   \eqref{b:21}. Integrating, as before, but with the improved bound \eqref{b:21}, we obtain 
	$$
	1-u(x)=\int_x^\infty u'(y) dy\leq C \int_x^\infty \f{1}{(1+|y|)^{\al+1}}dy \leq \f{C}{(1+|x|)^{\al}},
	$$
	which is \eqref{b:18}. 
	
	We now show that these are sharp bounds. Specifically, we move onto the asymptotic formula \eqref{260}. We may write 
	\begin{eqnarray*}
		u'(x) &=& 3 \int_{\rone} K_\al(x-y) (1-u^2(y)) u'(y) dy=3 K_\al(x) \int_{\rone} (1-u^2(y)) u'(y) dy+\\
		&+& 3 \int_{\rone} (K_\al(x-y)-K_\al(x))  (1-u^2(y)) u'(y) dy
	\end{eqnarray*}
	For the first term, we have by \eqref{g:20} and $\int_{\rone} (1-u^2(y)) u'(y) dy=(u-\f{u^3}{3})|_{-\infty}^\infty=\f{4}{3}$, 
	\begin{eqnarray*}
		 3 K_\al(x) \int_{\rone} (1-u^2(y)) u'(y) dy &=&  3\left(c_\al \f{1}{|x|^{1+\al}}+O(|x|^{-3})\right)\int_{\rone} (1-u^2(y)) u'(y) dy=\\
		 &=& \f{4c_\al}{|x|^{1+\al}}+O(|x|^{-3}).
	\end{eqnarray*}
where $c_\al$ is the constant in the asymptotic \eqref{g:20}. 

We will now show that the second term in the formula for $u'$ is also an error term. Indeed, we split the integration there in two regions: inside of the interval  $|y|<\f{|x|}{2}$ and outside of it $|y|\geq \f{|x|}{2}$. For the first region, using the estimates \eqref{g:45}, \eqref{b:18}, \eqref{b:21}, 
	\begin{eqnarray*}
& & 	\int_{|y|<\f{|x|}{2}} |K_\al(x-y)-K_\al(x)|  (1-u^2(y)) u'(y) dy   \\
& &\leq  \int_0^1 \int_{|y|<\f{|x|}{2}} 
	|K_\al'(x-\tau y)| |y| (1-u^2(y)) u'(y) dy d\tau \\
	& &\leq  C |x|^{-3} \int_\rone 
 |y| (1-u^2(y)) u'(y) dy= C |x|^{-3} \int_\rone  \f{|y|}{(1+|y|)^{2\al+1}} dy\leq C |x|^{-3}.
\end{eqnarray*}
For the other region, we have 
\begin{eqnarray*}
	  	\int_{|y|\geq \f{|x|}{2}} (K_\al(x-y)+K_\al(x)) (1-u^2(y)) u'(y) dy  &\leq &    C \int_{|y|\geq \f{|x|}{2}} \left(\f{K_\al(x-y)}{(1+|x|)^{2\al+1}}+ \f{|x|^{-3}}{(1+|y|)^{2\al+1}} \right)) dy\\
	& & \leq \f{C}{(1+|x|)^{1+2\al}}.
\end{eqnarray*}
Therefore, we conclude that 
$$
u'(x)=\f{4c_\al}{|x|^{1+\al}}+O(|x|^{-3}), |x|>>1,
$$
which is \eqref{260}. The formula \eqref{270} is obtained by integrating the exact asymptotic \eqref{260}, for $x>>1$, 
$$
u(x)=1-\int_x^\infty u'(y) dy= 1-\int_x^\infty \left(\f{4c_\al}{y^{1+\al}}+O(y^{-3})\right)dy= 1- \f{4c_\al}{\al} x^{-\al} +O(x^{-2}). 
$$


\end{proof}

\section{Asymptotic stability of the kinks: Proof of Theorem \ref{theo:20}} 
Let us start by reviewing the spectral information and the appropriate estimates for the linearized operator $\cl$. 
\subsection{Preparatory steps}

Recall that the operator $\cl$ is non-negative, with a simple and isolated eigenvalue at zero, $Ker[\cl]=span[\phi']$. Introduce the self-adjoint projection away from the kernel,  $P_{\{\phi'\}^\perp}$, defined via 
$$
P_{\{\phi'\}^\perp}=f- \|\phi'\|^{-2} \dpr{f}{\phi'}\phi'.
$$
One may define a self-adjoint restriction of $\cl$ away from the kernel, 
$$
\cl_0:=P_{\{\phi'\}^\perp} \cl P_{\{\phi'\}^\perp},\ \ \  \cl_0^*=\cl_0, 
$$ 
Due to \eqref{340}, we have that $\cl_0\geq \ka>0$. In addition, by the spectral theorem, the $C_0$ semigroup generated by $-\cl_0$ obeys the bound 
\begin{equation}
	\label{350} 
	\|e^{-t \cl_0} g\|_{L^2}\leq e^{- t \ka} \|g\|_{L^2}.
\end{equation}
We now use the non-negativity of $\cl$ to deduce an equivalence of the standard Sobolev spaces and the one provided by powers of $\cl$. To start with, the operator $(1+\cl)^\ga, \Re \ga>0$ is standard to define through the functional calculus for positive operators, with $D((1+\cl)^\ga))=H^{\al \ga}$. Since $(1+\cl)^\ga: H^{\al \ga}\to L^2$ is invertible, it is clear that 
$$
	\|g\|_{H^{\al \ga}}\sim \|(1+\cl)^\ga g\|_{L^2}.
$$
In fact, for the operator $\cl_0$, which is already invertible on $\{\phi'\}^\perp$, we have 
\begin{equation}
	\label{500} 
	\|g\|_{H^{\al \ga}}\sim \|\cl_0^\ga   g\|_{L^2}.
\end{equation}
for all functions $g:  P_{\{\phi'\}^\perp}  g=g$. So, one can establish the boundedness of $e^{-t \cl_0}$ between any two Sobolev spaces. More precisely, for any $s>0$, we claim that 
\begin{equation}
	\label{351} 
		\|e^{-t \cl_0} g\|_{H^s}\leq e^{- t \ka} \|g\|_{H^s}.
\end{equation}
Indeed, by the equivalence \eqref{500} and the commutation of any power of $\cl_0$ with the semigroup $e^{-t\cl_0}$, 
\begin{eqnarray*}
	\|e^{-t \cl_0} g\|_{H^s}\sim \|\cl_0^{\f{s}{\al}} e^{-t \cl_0} g\|_{L^2}=  \|e^{-t \cl_0} \cl_0^{\f{s}{\al}}  g\|_{L^2}\leq 
	e^{-t \ka} \| \cl_0^{\f{s}{\al}} g\|_{L^2}  \sim 	e^{-t \ka}  \| g\|_{H^s}
\end{eqnarray*}

Next, we state a standard lemma,   which guarantees a suitable decomposition of the solution $u$ of \eqref{10}. 
\begin{lemma}
	\label{le:45} 
	
Let   $T>0$ and $s\geq 0$.  Then, 	there exists $\eps_0>0$ and a constant $C$, so that for any 
$w\in C^1([0,T), H^s(\rone)): \|w\|_{C^1([0,T),H^s(\rone))}<\eps_0$, one can find uniquely   $v_w\in C^1([0,T), H^s)\cap \{\phi'\}^\perp$ and $\si_w\in C^1(\rone_+)$, so that 
\begin{eqnarray}
	\label{365} 
& & 	\phi(x)+w(t,x)=\phi(x+\si(t))+v(t,x) \\
	\label{370}
& & 	\|v\|_{C^1([0,T) H^s(\rone))}+\|\si\|_{C^1[0,T)}\leq C\eps_0
\end{eqnarray}
\end{lemma}
{\bf Remark:} The uniqueness of the decomposition is with respect to the smallness condition in \eqref{370}. 
We provide the relatively  standard proof of Lemma \ref{le:45} for completeness in the Appendix. 

\subsection{Setup for the proof}
We take initial data for \eqref{10} in the form $u_0=\phi+w_0$, where \\  $\|w_0\|_{H^\al(\rone)}<<1$. This means that we can setup a Cauchy problem for $w(t,x)=u(t,x)-\phi(x)$ in the form 
\begin{equation}
	\label{400} 
	w_t+D^\al w - w+  3\phi^2 w+3 \phi w^2+w^3=0, \ \ w(0,x)=w_0(x)
\end{equation}
We can rewrite it as an equivalent integral equation 
\begin{equation}
	\label{401} 
	w=e^{-t D^\al} w_0-\int_0^t e^{-(t-\tau) D^\al}[- w(\tau, \cdot)+  3\phi^2 w(\tau, \cdot)+3 \phi w^2(\tau, \cdot)+w^3(\tau, \cdot)] d\tau. 
\end{equation}
Standard estimates for the semigroup $e^{-t D^\al}$ show that there is a local solution in any $H^s, s\geq 0$ space (see for example the related \eqref{351}), local well-posedness for \eqref{400}, which guarantees 
$w\in C([0,T), H^s(\rone))$. Choose $s=\al$. From \eqref{400}, we can conclude that $w_t\in C^1([0,T), L^2(\rone))$, since 
$$
\|w_t\|_{L^2}\leq \|D^\al w\|_{L^2}+\|- w+  3\phi^2 w+3 \phi w^2+w^3\|_{L^2}\leq C \|w\|_{H^\al}(1+ \|w\|_{H^\al}^2),
$$
where we have used the Sobolev embedding $H^\al\hookrightarrow L^2\cap L^6(\rone)$. 

So, at least for a short time, 
$\|w(t, \cdot)\|_{H^\al(\rone)}$ will exist and it will be small, due to the fact that $\|w_0\|_{H^\alpha}$ is small. Also, $\|w_t\|_{L^2}$ is small as well. 
Thus, Lemma \ref{le:45} applies with $s=0$. 

Specifically, we may write the solution in the ansatz $y=\phi(x)+w(t,x)=\phi(x+\si(t))+v(t,x)$. In this new ansatz, we have the Cauchy problem, 
\begin{equation}
	\label{410} 
	 \left\{\begin{array}{l}
	 	v_t+D^\al v - v+  3\phi^2(x+\si(t)) v+3 \phi(x+\si(t)) v^2+v^3+\si'\phi'(x+\si(t))=0,\ \ v\perp \phi' \\
	  v(0,x)=w_0(x)=:v_0(x)\in H^\al(\rone)
	\end{array}
\right.
\end{equation}
We can further rewrite  as follows 
$$
		v_t+\cl v + 3(\phi^2(x+\si(t)) -\phi^2(x)) v+3 \phi(x+\si(t)) v^2+v^3+\si'(t) \phi'(x+\si(t))=0, 
$$
But, $\phi^2(x+\si(t)) -\phi^2(x)=
2 \si(t) \int_0^t \phi \phi'(x+\tau \si(t)) d\tau$. Altogether, 
\begin{equation}
	\label{420} 
		 \left\{\begin{array}{l}
	v_t+\cl v +\si'(t) \phi'(x+\si(t))+N=0, \ \ v\perp \phi'.\\ 
	 v(0,x)=v_0(x)
\end{array}
\right.
\end{equation}
where 
$$
N(t,x)=2 \si(t) v(x) \int_0^t \phi \phi'(x+\tau \si(t)) d\tau+3 \phi(x+\si(t)) v^2(x)+v^3(x).
$$
Thus, our task is to show persistence of regularity for  the Cauchy problem \eqref{420} with small data.  
\subsection{Persistence of  regularity}
We state the main result in the following Proposition. 
\begin{proposition}
	\label{prop:129}
	There exists $\eps_0>0$ and a constant $C$, so that whenever $\|v_0\|_{H^\al(\rone)}<\eps$, the Cauchy problem \eqref{420} is globally well-posed and 
	\begin{eqnarray}
		\label{450} 
		\|v(t,\cdot)\|_{H^\al(\rone)}\leq C \eps e^{-\ka t} \\
		\label{460} 
		|\si'(t)|\leq C \eps  e^{-2 \ka t}, |\si(t)|\leq C \eps.
	\end{eqnarray}
\end{proposition}

\begin{proof}
We decompose  \eqref{420} along the directions $\{\phi'\}^\perp$ and along $\phi'$. Taking the inner product of \eqref{420} with $\phi'$ yields a non-linear equation for $\si'$,  
	\begin{equation}
		\label{s:10} 
	-\dpr{N}{\phi'} = 	\si'(t) \dpr{\phi'(\cdot)}{\phi'(\cdot+\si(t))}=	\si'(t) \left(\|\phi'\|^2+ \dpr{\phi'(\cdot)}{\phi'(\cdot+\si(t))-\phi'(\cdot)} \right)
\end{equation}
 Note that 
 $$
 |\dpr{\phi'(\cdot)}{\phi'(\cdot+\si(t))-\phi'(\cdot)}|\leq |\si(t)| \|\phi'\|_{L^1} \|\phi''\|_{L^\infty}.
$$
where $\si(t)$  is expected to be small along the evolution. Under the appropriate smallness assumptions on $\si(t)$ then, 
$$
\|\phi'\|^2+ \dpr{\phi'(\cdot)}{\phi'(\cdot+\si(t))-\phi'(\cdot)} \geq \f{\|\phi'\|^2}{2}, 
$$
 and so 
one can properly solve \eqref{s:10}, 
	\begin{equation}
		\label{430} 
 	\si'(t) = \f{-\dpr{N(t, \cdot)}{\phi'}}{\|\phi'\|^2+  \dpr{\phi'(\cdot)}{\phi'(\cdot+\si(t))-\phi'(\cdot)}  }=:\La_\si[v,\si]
	\end{equation}
Next, 	as $P_{\{\phi'\}^\perp}$  is a projection, projecting perpendicular to an eigenfunction, it commutes with $\cl$, so $P_{\{\phi'\}^\perp} \cl=P_{\{\phi'\}^\perp} \cl P_{\{\phi'\}^\perp}=\cl_0$, we have 
\begin{equation}
	\label{470} 
	v_t+\cl_0 v +\si'(t) P_{\{\phi'\}^\perp}[\phi'(\cdot+\si(t))]+P_{\{\phi'\}^\perp}[N]=0.
\end{equation}
	Applying the Duhamel's formula for \eqref{470}, we arrive at the corresponding  integral equation, 
	\begin{equation}
		\label{480} 
		v=e^{-t \cl_0} v_0 +\int_0^t e^{-(t-s) \cl_0}\left[ \si'(s) P_{\{\phi'\}^\perp}[\phi'(\cdot+\si(s))]+P_{\{\phi'\}^\perp}[N]\right] ds=:\La_v[v,\si]
	\end{equation}

	We now show the required persistence of regularity properties. Specifically, we claim that both $v_0$ and $\si'(t)$  will remain small  and exhibit exponential decay $e^{-\ka t}$, while  $\si(t)$ will just remain small (without the exponential decay of course), under \eqref{430} and \eqref{480}.  
	
	Specifically, assume that  for some small enough  $\eps$, we have 
	\begin{equation}
		\label{510} 
			\|v(t, \cdot)\|_{L^2}\leq \eps e^{-\ka t}, 
		|\si'(t)|\leq C \eps e^{-2 \ka t}, |\si(t)|<C \eps
	\end{equation}
	for all $t\in (0,T)$. We will show that the estimates \eqref{510} persist going forward under the time evolution.
	
	 To this end, note that, as we pointed out above, the denominator of \eqref{430} is at least $\f{\|\phi\|^2}{2}$, 
	 so that 
	\begin{eqnarray*}
		|\La_\si[v, \si](t)| &\leq &  \f{2}{\|\phi'\|^2}|\dpr{N}{\phi'}|\leq \f{2}{\|\phi'\|}\|N\|_{L^2} \leq C (|\si(t)| \|v\|_{L^2} \|\phi\phi'\|_{L^\infty}+\|\phi\|_{L^\infty} \|v\|_{L^4}^2+ \|v\|_{L^3}^3)\\
		&\leq & C(\eps^2  e^{-3\ka t} + \eps^2 e^{-2\ka t}+\eps^3 e^{-3\ka t}).
	\end{eqnarray*}
where we have used the {\it a priori} bounds \eqref{510} as well as the Sobolev embedding  $H^s(\rone)\hookrightarrow L^4(\rone), L^3(\rone)$. Clearly, such an expression on the right guarantees the propagation of the bound. 

For $v$, we have by taking $H^\al(\rone)$ norms in \eqref{480}, 
	\begin{eqnarray*}
\|\La_v(v, \si)\|_{H^\al(\rone)}& \leq &  \|e^{-t \cl_0} v_0\|_{H^\al(\rone)}+\int_0^t |\si'(\tau)| \|e^{-(t-\tau) \cl_0}  P_{\{\phi'\}^\perp}[\phi'(\cdot+\si(\tau))]\|_{H^\al} d\tau +\\
&+& \int_0^t \| e^{-(t-\tau) \cl_0} P_{\{\phi'\}^\perp}[N(\tau, \cdot)]\|_{H^\al} d\tau
\end{eqnarray*}
{\bf Estimate on initial data:} \\ 
	By \eqref{351}, with $s=\al$, 
	$$
	 \|e^{-t \cl_0} v_0\|_{H^\al(\rone)}\leq   C_0 e^{- \ka t}  \|v_0\|_{H^\al(\rone)} 
	$$
	{\bf Estimate on the term $ P_{\{\phi'\}^\perp}[\phi'(\cdot+\si(\tau))]$:}
	$$
	\|e^{-(t-\tau) \cl_0}  P_{\{\phi'\}^\perp}[\phi'(\cdot+\si(\tau))]\|_{H^\al} \ 
	 \leq C e^{-\ka (t-\tau)} \|P_{\{\phi'\}^\perp} [\phi'(\cdot+\si(\tau))]\|_{H^\al}.
$$
However, 
$$
	 P_{\{\phi'\}^\perp} [\phi'(\cdot+\si(\tau))] (x)= \phi'(x+\si(t)) - \phi'(x) - \|\phi'\|^{-2}\dpr{\phi'(\cdot)}{\phi'(\cdot+\si(\tau))-\phi'(\cdot)},
$$
whence 
$$
\|P_{\{\phi'\}^\perp} [\phi'(\cdot+\si(\tau))]\|_{H^\al}\leq C \|\phi'(\cdot+\si(t))-\phi'(\cdot)\|_{H^\al}\leq C |\si(\tau)|\leq C \eps.
$$
		{\bf Estimate on the nonlinear term $P_{\{\phi'\}^\perp}[N]$:}
		We have by the estimates on the semigroup in Sobolev spaces, \eqref{351} and the product estimates in $H^\al$, \eqref{li:10}, 
		\begin{eqnarray*}
			\| e^{-(t-\tau) \cl_0} P_{\{\phi'\}^\perp}[N(\tau, \cdot)]\|_{H^\al} &\leq & 	C e^{-\ka(t-\tau)} 
			\|N(\tau, \cdot)\|_{H^\al}\\
			&\leq &  C(|\si(\tau)| \|v(\tau)\|_{H^\al} + \|v(\tau)\|_{H^\al}^2+\|v(\tau)\|_{H^\al}^3)\leq \\
			&\leq & C\eps^2 e^{-2\ka\tau}.
		\end{eqnarray*}
	Collecting all of these entries in the estimate for $	\|\La_v[v, \si](t)\|_{H^\al} $ allows us to conclude
	\begin{eqnarray*}
		\|\La_v[v, \si](t)\|_{H^\al} \leq C e^{-\ka t} \|v_0\|_{H^\al}+ 
		\int_0^t e^{-\ka(t-\tau)} \eps^2 e^{-2\ka \tau}   d\tau \leq C_0 \eps e^{-\ka t} + C\eps^2  e^{-\ka t}. 
	\end{eqnarray*} 
	This shows that for small enough $\eps$, we can ensure that 
	$$
	\|v(t, \cdot)\|_{H^\al(\rone)}\leq 2C_0 \eps e^{-\ka t}, \ \ |\si'(t)| \leq C\eps e^{-2\ka t},|\si(t)| \leq C\eps. 
	$$

\end{proof}


\section{Kinks for the fractional wave equation: Proof of Theorem \ref{theo:30}} 
We start our considerations by indicating the main steps in the existence proof for \eqref{35}. Fix $c: c\in (-1,1)$. 
\subsection{Existence of the kinks}
Most of the existence results in this section follow verbatim the proofs in the case of \eqref{30}. Indeed, we set up our variational problem, similar to \eqref{48}, 
 $$
 \f{1}{2} \left((1-c^2) \|u'\|^2+ \|D^{\al/2} u\|^2\right) + \f{1}{4} \int_{\rone} (1-u^2)^2 dx\to \min.
 $$
Then, we pivot to the following version of the approximate problem \eqref{50}, 
 $$
  \f{1}{2} \left((1-c^2) \|u'\|^2+ \|D^{\al/2} u\|^2\right) + \epsilon \int_\rone (u-W)^2 \ln(e+x^2) dx + \f{1}{4} \int_{\rone} (u^2-1)^2 dx\to \min.
 $$
The arguments from here are identical - one proves (with the additional harmless term $-(1-c^2)\p_{xx}$) a version of Proposition \ref{prop:10}, which establishes approximate kink solutions. The appropriate version of Proposition \ref{prop:20}, which proceeds in an identical manner, establishes  uniform estimates for $q _\epsilon$ in $H^{\al/2}$ and $v_\eps$ in $L^\infty\cap \dot{H}^{\al/2}$.  

Then, a version of Proposition \ref{prop:40} establishes the validity of the limit $\lim_{\epsilon\to 0+} v_\eps=v$, together with the smoothness and decay property  $u\in L^\infty(\rone)\cap \dot{H}^{\f{\al}{2}}(\rone)\cap_{l=1}^\infty \dot{H}^l(\rone), u'\in \cap_{l=0}^\infty H^l(\rone)$,  together with the non-negativity of the linearized operator $\cl=-(1-c^2)\p_{xx}+D^\al +2 - 3(1-u^2)\geq 0$, with simple and single eigenvalue at zero, spanned by $u'$. Note that in this argument, one needs  that $\cl$ has the Perron-Frobenius property, which is equivalent, \cite{FL},  to the positivity of the kernel of the semigroup $e^{-t \cl}$. Standard arguments, see for example \cite{FL}, reduce this to checking that $e^{-t(-(1-c^2)\p_{xx}+D^\al)}$ has positive kernel. But since $\p_{xx}$ commutes with $D^\al$, we have that 
$$
e^{-t(-(1-c^2)\p_{xx}+D^\al)}=e^{t(1-c^2)\p_{xx}} e^{-tD^\al}.
$$
 So, it follows that the kernel of $e^{-t(-(1-c^2)\p_{xx}+D^\al)}$ is positive as a convolution of two positive kernels: the standard heat kernel  and the fractional heat kernel (recall $\al<2$), see for example  \cite{FL} for this last fact. 
 
 Next, we address the decay rates. 
 \subsection{Decay rates}
The proof of the decay rates proceeds identically to Section \ref{sec:4}, once we have the appropriate version of Lemma \ref{g:10}. 
\begin{lemma}
	\label{g:11} 
	Let $\al\in (1,2)$. Then, the resolvent operator $(-(1-c^2)\p_{xx}+D^\al+2)^{-1}$  is represented as 
	$$
	(-(1-c^2)\p_{xx}+D^\al+2)^{-1} F(x)=\int_{-\infty}^\infty \tilde{K}_\al(x-y) F(y) dy, 
	$$
	with a positive kernel $\tilde{K}_\al$ with 
	\begin{eqnarray}
		\label{g:251} 
		& & 0<\tilde{K}_\al(x)\leq \f{C_\al}{(1+|x|)^{1+\al}}, \\
		\label{g:201} 
	& & 	\tilde{K}_\al(x)= \f{2^{\al-4} \al(\al-1) \Ga\left(\f{\al-1}{2}\right)}{\sqrt{\pi} \Ga\left(\f{2-\al}{2}\right)}     |x|^{-1-\al}+O(|x|^{-3}), \ \  |x|>1
	\end{eqnarray}
	Also 
	\begin{equation}
		\label{g:451} 
		|\tilde{K}_\al'(x)|\leq C |x|^{-3}, |x|>>1.
	\end{equation}
\end{lemma}

\subsection{Spectral stability of the kinks $\phi_c, c\in (-1,1)$}
In this section, we need to show that the eigenvalue problem \eqref{lin:30} has only non-trivial solutions if $\Re \la=0$. It is worth noting that the said spectral stability follows from a general instability index theory, once we take into account the positivity of the linearized operator $\ch:=\begin{pmatrix}
	\cl & 0 \\ 0 & Id
\end{pmatrix}$.  

Nevertheless, we provide an instructive direct proof of this fact. 
Suppose that \eqref{lin:30} has  a solution $\vec{f}\neq 0$  for some $\la:\Re\la\neq 0$. 
	$$
	\begin{pmatrix}
	0 & Id \\  -Id & 2c\p_x 
\end{pmatrix} \begin{pmatrix}
	\cl & 0 \\ 0 & Id
\end{pmatrix} \vec{f}=\la \vec{f},
$$
and apply the operator $\cj:=\begin{pmatrix}
	2c\p_x & -Id \\  Id & 0
\end{pmatrix}=\begin{pmatrix}
0 & Id \\  -Id & 2c\p_x 
\end{pmatrix}^{-1}$ to both sides. We get 
\begin{equation}
	\label{800}
	\begin{pmatrix}
		\cl & 0 \\ 0 & Id
	\end{pmatrix} \vec{f} = \la \cj \vec{f}.
\end{equation}
 Taking dot product with $\vec{f}$ yields 
$$
 \dpr{\ch \vec{f}}{\vec{f}}=\la \dpr{\cj \vec{f}}{\vec{f}}.
$$
Noting that $\ch^*=\ch$, while $\cj^*=-\cj$, we obtain 
\begin{eqnarray*}
& & 	\dpr{\ch \vec{f}}{\vec{f}}=\dpr{ \vec{f}}{\ch \vec{f}}=\overline{\dpr{\ch  \vec{f}}{\vec{f}}}\\
& & 	\dpr{\cj \vec{f}}{\vec{f}}=-\dpr{ \vec{f}}{\cj \vec{f}}=-\overline{\dpr{\cj \vec{f}}{\vec{f}}}
\end{eqnarray*}
whence $\dpr{\cj \vec{f}}{\vec{f}}$ is purely imaginary, whereas $\dpr{\ch \vec{f}}{\vec{f}}$ is real. We claim that $\dpr{\cj \vec{f}}{\vec{f}}\neq 0$. Supposing for a contradiction that 
$ \dpr{\cj \vec{f}}{\vec{f}}=0$.  It follows that 
$$
0=\dpr{\ch \vec{f}}{\vec{f}}=\dpr{\cl f_1}{f_1}+\|f_2\|^2,
$$
which since $\cl\geq 0$ implies that $f_2=0$. Going back to the second equation in \eqref{800}, this means that $0=\la f_1$, so $f_1=0$, a trivial solution, thus a contradiction. 
This means that $\dpr{\cj \vec{f}}{\vec{f}}\neq 0$. But then 
$$
\la=\f{\dpr{\ch \vec{f}}{\vec{f}}}{\dpr{\cj \vec{f}}{\vec{f}}}
$$
is purely imaginary as ratio of real and purely imaginary. A contradiction, which implies that the spectrum is only purely imaginary, i.e. 
$$
\si(\begin{pmatrix}
	0 & Id \\  -Id & 2c\p_x 
\end{pmatrix} \begin{pmatrix}
	\cl & 0 \\ 0 & Id
\end{pmatrix})\subset \{\la\in {\mathbb C}:  \Re\la=0\}. 
$$

\section{Existence of kinks and asymptotic in the super-Laplacian case: Proof of Theorem \ref{theo:40}}
The existence proof for the case $\al\in (1,2)$ goes without any difficulties for the case $\al\in (2,4)$. Specifically, Proposition \ref{prop:10} provides solutions to the approximate minimization problem, Proposition \ref{prop:20} provides uniform {\it a priori} estimates in $0<\epsilon<1$, while Proposition \ref{prop:40} justifies  the limits $u_\eps$ toward $u$, $u\in L^\infty\cap H^\infty(\rone), u'\in H^\infty(\rone)$, the limit $\lim_{x\to \pm \infty} u(x)=\pm 1$ together with the positivity of the linearized operator $\cl$ and , along with $\cl[u']=0$. 
\subsection{Sharp asymptotics of the Green's function}

Regarding the asymptotics of $u(x), |x|>>1$, the key is again at the kernel of the resolvent $(D^\al+2)^{-1}$. This kernel can no longer be expected to be positive (and in fact, it is not). As we have learned, the rate of decay of the derivative $u'$ is inherited from the decay of the kernel of $(D^\al+2)^{-1}$. We have the following version of Lemma \ref{g:10}. 
\begin{lemma}
	\label{g:12} 
	Let $\al\in (2,4)$. Then, the resolvent operator $(D^\al+2)^{-1}$  is represented as 
	$$
	(D^\al+2)^{-1} F(x)=\int_{-\infty}^\infty K_\al(x-y) F(y) dy, 
	$$
	for which we have  that 
	\begin{equation}
		\label{fr:50} 
		\|K_\al(x)\|_{L^\infty}\leq C<\infty,
	\end{equation}
and the asymptotic formula, 
	\begin{equation}
		\label{fr:60} 
		K_\al(x)= -\f{2^{\al-5}\al(\al-1)(\al-2)}{\sqrt{\pi}} 
		\f{\Ga\left(\f{\al-1}{2}\right)}{\Ga\left(\f{4-\al}{2}\right)}|x|^{-1-\al}+O(|x|^{-5}). 
	\end{equation}
	Also 
	\begin{equation}
		\label{fr:70} 
		|K_\al'(x)|\leq C |x|^{-5}, |x|>>1.
	\end{equation}
    For $\al=2, 4$, we have the precise formulas, using Mathematica, 
    \begin{eqnarray}
        \label{p:10}
        K_2(x) &=& \f{e^{-\sqrt{2}|x|}}{2\sqrt{2}}\\
        \label{p:11}
        K_4(x) &=& \sqrt[4]{2} \f{e^{-\f{|x|}{\sqrt[4]{2}}}}{4 }\sin\left(\f{|x|}{\sqrt[4]{2}}+\f{\pi}{4}\right).
    \end{eqnarray}
\end{lemma}
We postpone the proof of Lemma \ref{g:12} for the Appendix. 
Some remarks are however in order: 
{\bf Remarks:} 
\begin{itemize}
    \item We note  that 
in \cite{LePel}, the authors established a related result for the kernel of the {\it periodic} Green's function. In it, they have conjectured that the kernel $K_\al$ of the Green's function on $\rone$, when $\al\in (2,4)$, is sign-changing for large $|x|$. The asymptotic formula \eqref{fr:60} establishes precisely this:  it shows that the kernel $K_\al$ is negative for large $|x|$, while  $K_\al(0)=\int (2+(2\pi|\xi|)^\al)^{-1} d\xi>0$. In fact, we also establish another observation made in \cite{LePel}, that the kernel $K_4$ vanishes at infinitely many points. This is also in line with numerical explorations,
such as those of~\cite{DeckerCNSNS2021}.
\item Another notable fact is the exponential decay at $\al=2,4$ (and in fact for all even integers $\al$), see \eqref{p:10}, \eqref{p:11}. One might wonder about the asymptotics in \eqref{fr:60} viz. a viz the \eqref{p:10}, \eqref{p:11}. One can see that the asymptotic formulas \eqref{fr:60} collapse as $\al\to 2+$ and $\al\to 4-$, as the coefficient in front of $|x|^{-1-\al}$ tends to zero. This allows for the exponential rates to take over at these particular values. 
\end{itemize}

 We are now ready for the proof of the asymptotics of the kink. 

\subsection{Sharp asymptotics for the kink solution}
As expected the proof follows along the ideas of Proposition \ref{prop:50}. The details are stated in the next Proposition. 
\begin{proposition}
	\label{prop:55} 
	 For the solution $u$ constructed in Theorem \ref{theo:40}, 
	  there is the sharp decay rate 
	\begin{equation}
		\label{c:182} 
		|u(x)-sgn(x)|\leq C (1+|x|)^{-\al}.
	\end{equation}
In addition, 
	\begin{equation}
		\label{c:21} 
		|u'(x)|< \f{C}{(1+|x|)^{1+\al}}
	\end{equation}
	Moreover,   
	\begin{eqnarray}
		\label{c:26} 
		u'(x) &=&-\f{2^{\al-3}\al(\al-1)(\al-2)}{\sqrt{\pi}} 
		\f{\Ga\left(\f{\al-1}{2}\right)}{\Ga\left(\f{4-\al}{2}\right)}|x|^{-1-\al}+O(|x|^{-5}), x>>1  \\
		\label{c:27} 
		u(x) &=&  1+ \f{2^{\al-3}(\al-1)(\al-2)}{\sqrt{\pi}} 
		\f{\Ga\left(\f{\al-1}{2}\right)}{\Ga\left(\f{4-\al}{2}\right)}|x|^{-\al}+O(|x|^{-4}), x>>1
	\end{eqnarray}
\end{proposition}
\begin{proof}
	Even though the proof follows almost verbatim the proof of Proposition \ref{prop:50}, we indicate the main points. 
	
	We start by writing  an integral equation for $u'$. Recall however that $u'$ is no longer positive, nor do we expect $|u(x)|<1$. In fact,  we show rigorously,  that this is not the case for large $|x|$.   
    After an estimate with absolute value, we obtain  
	\begin{eqnarray*}
		|u'(x)| &=& 3|(D^\al+2)^{-1}[(1-u^2)u']|=3\int_\rone |K_\al(x-y)||1-u^2(y)| |u'(y)| dy\leq \\
		&\leq & C_\al \int_\rone \f{1}{(1+|x-y|)^{1+\al}}|1-u^2(y) |u'(y)| dy,
	\end{eqnarray*}
As $u'\in H^\infty(\rone)$, $\lim_{x\to \infty} u'(x)=0$. Also, $\lim_{x\to \infty} u(x)=1$, so for all $\delta>0$, there is $k_0: |1-u^2(x)|<\de$, whenever $x>2^{k_0}$. 
Define again,  $a_k:=\max_{x\geq 2^k} |u'(x)|$. Taking $\max_{x\geq 2^k}$ in the inequality for $|u'(x)|$ above leads, as before, to \eqref{200}, namely 
$$
a_k \leq C_\al \de a_{k-1} +C_\al 2^{-k(\al+\f{1}{2})}.
$$
Bootstrapping this estimate, identically as in the proof of Proposition \ref{prop:50} leads to the intermediate bound \eqref{220}, $|u'(x)|\leq C(1+|x|)^{-\al-\f{1}{2}}$, which falls short of the optimal one \eqref{c:21}. However, it implies 
$$
|1-u(x)|\leq \int_x^\infty |u'(y)| dy\leq \frac{C}{(1+|x|)^{\al-\f{1}{2}}}. 
$$
Going back to the integral inequality for $|u'(x)|$ with these improved bounds, for 
$|u'(y)|, |1-u(y)|$ leads to the estimate 
$$
|u'(x)|\leq C \int_{\rone} \f{1}{(1+|x-y|)^{1+\al}} \f{1}{(1+|y|)^{2\al}} dy\leq C (1+|x|)^{-1-\al},
$$
whence the estimate \eqref{c:182} follows immediately. 

For the leading order term in $u'$, let $x>>1$ and  write as in Proposition \ref{prop:50}, 
	\begin{eqnarray*}
	u'(x) &=& 3 \int_{\rone} K_\al(x-y) (1-u^2(y)) u'(y) dy=3 K_\al(x) \int_{\rone} (1-u^2(y)) u'(y) dy+\\
	&+& 3 \int_{\rone} (K_\al(x-y)-K_\al(x))  (1-u^2(y)) u'(y) dy
\end{eqnarray*}
The first term has, due to the asymptotic \eqref{fr:60}, 
\begin{eqnarray*}
	 3 K_\al(x) \int_{\rone} (1-u^2(y)) u'(y) dy &=& 3 K_\al(x) \left(u-\f{u^3}{3}\right)|_{-\infty}^\infty=4 K_\al(x)= \\
	&=& -\f{2^{\al-3}\al(\al-1)(\al-2)}{\sqrt{\pi}} 
	\f{\Ga\left(\f{\al-1}{2}\right)}{\Ga\left(\f{4-\al}{2}\right)}|x|^{-1-\al}+O(|x|^{-5}). 
\end{eqnarray*}
The second term is an error term. Indeed, splitting the integrals 
in the regions $|y|<\f{|x|}{2}$, $|y|>\f{|x|}{2}$, we have by \eqref{c:21}, \eqref{c:182} and \eqref{fr:70}, 
$$
\int_{|y|<\f{|x|}{2}} |K_\al(x-y)-K_\al(x)|  |1-u^2(y)| |u'(y)| dy\leq C |x|^{-5} \int_{\rone} 
\f{|y|}{(1+|y|)^{2\al+1}} dy\leq C |x|^{-5}
$$
For the outer region, we estimate by 
\begin{eqnarray*}
	\int_{|y|>\f{|x|}{2}} [|K_\al(x-y)|+|K_\al(x)|]  |1-u^2(y)| |u'(y)| dy 	&\leq & 
C 	\int_{|y|>\f{|x|}{2}} \f{|K_\al(x-y)|+|x|^{-1-\al}}{(1+|y|)^{2\al+1}}dy \leq \\
&\leq &  \f{C}{(1+|x|)^{1+2\al}}. 
\end{eqnarray*}
Thus, we conclude \eqref{c:26}. For \eqref{c:27}, we have 
\begin{eqnarray*}
	u(x)=1-\int_x^\infty u'(y) dy= 1-\int_x^\infty \left(-\f{c_\al}{y^{1+\al}}+O(y^{-5})\right)dy= 1+ \f{c_\al}{\al} x^{-\al} +O(x^{-4}), 
\end{eqnarray*}
where $c_\al$ is the constant in \eqref{c:26}. Formula \eqref{c:27} follows.

\end{proof}

\section{Conclusions and Outlook}

In the present work we revisited the features of models of the
$\phi^4$ class (either of the parabolic type, as in the Allen-Cahn
case~\cite{Bartels2015}, or of the wave type~\cite{p4book}) in the
presence of a fractional spatial derivative of the Riesz type
with variable exponent $\alpha$. We presented a fundamental
distinction between the monotonic nature of the kink nonlinear
waveform in the case of $1< \alpha \leq  2$ and the non-monotonic,
unity-crossing case of $2<\alpha<4$. In each one of these
cases, we presented rigorous asymptotic estimates
of the tails of the kinks with the corresponding $\alpha$-dependent
power laws and the appropriate sharp prefactors based on our
analysis. These estimates were firmly corroborated by our
numerical computations in each of the two above scenaria. Additionally,
we offered a proof of the uniqueness of the resulting kinks, 
contigent on an eigenvalue estimate that was subsequently
confirmed numerically. Hence, we have ensured that the relevant
coherent structures are indeed unique. Moreover, we have
illustrated their asymptotic stability in the parabolic
model case example. In the wave setting, we have showcased the
spectral stability of the relevant states.

On the other hand, there are numerous emerging directions
that merit further exploration. In the realm of single
structures, while the understanding of the wave problem may be
significantly advanced, fewer results are rigorously available
for the corresponding (nonlinear) Schr{\"o}dinger case. Recall that the
latter is the one which is more directly experimentally
accessible, as per the recent work of~\cite{hoang2024observationfractionalevolutionnonlinear}.
While the two may share, practically, the same steady state problem,
the stability results need to be suitably adapted to the latter. 
Furthermore, earlier works such as that of~\cite{decker2024fractionalsolitonshomotopiccontinuation} have
examined the possibility of multi-solitary wave bound states
in the realm of $2 < \alpha < 4$. Presently, we are not
aware of any rigorous results, potentially validating 
this numerical picture. Of course, all of the above
constitute considerations chiefly applicable in one
spatial dimension. Naturally, extending relevant findings
to higher dimensional settings is an interesting topic in its own
right. Studies along these veins are presently in progress and will
be reported in future publications.


\appendix

\section{On the fractional derivative of Schwartz functions}
\begin{lemma}
	\label{app:10}
	Let $\Psi$ be a Schwartz function on $\rone$, while  $s>0$. Then, 
	\begin{equation}
		\label{ap:10} 
		|D^s \Psi(x)|\leq C_{s,\Psi} (1+|x|)^{-1-s}.
	\end{equation}
\end{lemma}
\begin{proof}
	For the $L^\infty$ bounds, we have 
	$$
	\sup_x |D^s \Psi(x)|\leq (2\pi)^s \int_\rone |\xi|^s |\hat{\Psi}(\xi)| d\xi<C. 
	$$
	Assume now that $|x|>1$. By the inversion formula for FT, 
	\begin{eqnarray*}
		(2\pi)^{-s} D^s \Psi(x) &=&  \int_{-\infty}^\infty |\xi|^s \hat{\Psi}(\xi) e^{2\pi i x \xi} d\xi=   \\
		&=& 
		\int_{-\infty}^\infty |\xi|^s \hat{\Psi}(\xi) (1-\chi(\xi)) e^{2\pi i x \xi} d\xi+ 
		\int_{-\infty}^\infty |\xi|^s \hat{\Psi}(\xi) \chi(\xi) e^{2\pi i x \xi} d\xi 
	\end{eqnarray*}
	Clearly, the first integral represents the inverse FT of a Schwartz function, so it has power decay of all orders, hence \eqref{ap:10} holds. 
	For the second integral,  introduce a partition of unity 
	$$
	\chi(z)=\chi(z)\chi(2^{-3} z)=\chi(z) \sum_{j=-3}^\infty\psi(2^j \xi),
	$$
	 so that 
	 	\begin{eqnarray*} 
	 	& & \int_{-\infty}^\infty |\xi|^s \hat{\Psi}(\xi) \chi(\xi) e^{2\pi i x \xi} d\xi= \sum_{j=0}^\infty 	\int_{\rone}   |\xi|^s \psi(2^j \xi)  \hat{\Psi}(\xi) \chi(\xi) e^{2\pi i x \xi} d\xi=\\
	 	&=& \sum_{j=-3}^\infty 	2^{-js} \int_{\rone}  \psi_s(2^j \xi)  \hat{\Psi}(\xi) \chi(\xi) e^{2\pi i x \xi} d\xi=
	 	\left(\sum_{j=-3}^\infty 	2^{-j(1+s)} \check{\psi}_s(2^{-j}\cdot)\right)*\Psi*\check{\chi}.
	 \end{eqnarray*}
	where $\psi_s(z):=\psi(z) |z|^s$ is a Schwartz function as well.   $\Psi*\check{\chi}$ is Schwartz, so it has power decay of all orders.  We have that for all integers $N$ (one should select say $N=[s]+10$), 
	$$
	\sum_{j=-3}^\infty 	
	2^{-j(1+s)} |\check{\psi}_s(2^{-j}x)|\leq C \sum_{j=-3}^\infty 	
	2^{-j(1+s)}  \left(\f{1}{1+\f{|x|}{2^j}}\right)^N\leq C |x|^{-1-s}.
	$$
	The result now follows from \eqref{250}. 
\end{proof}
	\section{Decay rate and asymptotics of the fractional Green's function: Proof of Lemma \ref{g:10} and Lemma \ref{g:11}}
	
 We present the proof for Lemma \ref{g:10}, while the proof for Lemma \ref{g:11} requires only minor modifications, which we indicate at the end. 
 
		The positivity of the kernel follows from the Laplace transform representation of the resolvent 
		$$
		(D^\al+2)^{-1}=\int_0^\infty e^{-tD^\al} e^{-2t}  dt, 
		$$
		and the fact that the $C_0$ semigroup $e^{-tD^\al}$ generated by the fractional Laplacian $-D^\alpha, \alpha\in (0,2)$ is a positive semigroup, \cite{FL}. 
		The $L^\infty$ bound is immediate, due to the estimate 
		$$
		\sup_x |K_\al(x)|\leq \int_\rone \f{1}{2+(2\pi |\xi|)^\al} d\xi <\infty.
		$$
		So, assume for the rest of the proof that $x>1$  - the case $x<-1$ is treated similarly as  $K_\al$ is an even kernel. 
		First, decompose 
		\begin{eqnarray*}
			K_\al(x) &=& \int_\rone \f{1}{2+(2\pi |\xi|)^\al} e^{2\pi i x\xi} d\xi=\\
			&=& \int_\rone \f{1}{2+(2\pi |\xi|)^\al} \chi(16 \xi)e^{2\pi i x\xi} d\xi+\int_\rone \f{1}{2+(2\pi |\xi|)^\al} (1-\chi(16 \xi))e^{2\pi i x\xi} d\xi=:K_{\al,1}(x)+K_{\al,2}(x)
		\end{eqnarray*}
		Clearly, the second integral represents the inverse Fourier transform of the   function 
		$$
		\xi\to \f{1}{2+(2\pi |\xi|)^\al} (1-\chi(16 \xi))\in C^\infty(\rone), 
		$$ 
		and in addition, its derivatives are all integrable. Thus, 
		$|K_{\al,2}(x)|\leq C_N (1+|x|)^{-N}$ for each $N$. So, it remains to analyze the first term. 
		 Due to the even integrand and support consideration, note that on the support of $\chi(16\xi)$, we have $|\xi|<1/8$. So,  $\f{(2\pi |\xi|)^\al}{2}<1$, 
		and we can represent by power  series, 
		$$
		 \f{1}{2+(2\pi |\xi|)^\al} \chi(16 \xi) = 
		 \left(\sum_{k=0}^\infty (-1)^k \f{(2\pi)^{k\al}}{2^{k+1}} |\xi|^{k\al}\right)\chi(16 \xi)=\f{1}{2}\chi(16 \xi) - \f{(2\pi)^{\al}}{4} 
		 |\xi|^{\al}\chi(16 \xi)+\ldots 
		$$
		It follows that 
		\begin{eqnarray}
			\label{m:10}
			K_{\al,2}(x) &=& 
			 \int_\rone \f{1}{2+(2\pi |\xi|)^\al} \chi(16 \xi)e^{2\pi i x\xi} d\xi =
			 \f{1}{2}  \int_\rone   \chi(16 \xi)e^{2\pi i x\xi} d\xi \\
			  \nonumber 
			  & - &  \f{(2\pi)^{\al}}{4} 
			\int_\rone   |\xi|^{\al}\chi(16 \xi) e^{2\pi i x\xi} d\xi + 
			\sum_{k=2}^\infty (-1)^k \f{(2\pi)^{k\al}}{2^{k+1}}\int_\rone |\xi|^{k\al} \chi(16 \xi)e^{2\pi i x\xi} d\xi.
		\end{eqnarray}
		The first integral represents the inverse FT of a Schwartz function, whence it has all possible power decay. We claim that the sum also has a faster decay than the one announced in \eqref{g:25}. 
		Indeed, recalling that $k\geq 2$, after reorganizing and integrating by part three times, we obtain 
		\begin{eqnarray*}
			E_k(x) &:= & \left|\int_\rone |\xi|^{k\al} \chi(16 \xi)e^{2\pi i x\xi} d\xi\right|=
			16^{-k\al-1} 	\left|\int_\rone |\eta|^{k\al} \chi(\eta)e^{\pi i x\f{\eta}{8}} d\eta\right|\\
			&\leq &  C 16^{-k\al} 	\left|\int_0^\infty \eta^{k\al} \chi(\eta)\cos(\pi  x\f{\eta}{8})d\eta\right|\leq C \f{16^{-k\al}}{x^3} 
			\left|\int_0^\infty  (\eta^{k\al} \chi(\eta))''' \sin(\pi  x\f{\eta}{8}) d\eta\right| \\
			&\leq & C \f{16^{-k\al}}{x^3} \int_\rone |(\eta^{k\al} \chi(\eta))'''| d\eta.
		\end{eqnarray*}
	Note that this last integral is integrable near zero as $|(\eta^{k\al} \chi(\eta))'''| \sim |\eta|^{k\al-3}, |\eta|<<1$ and $k\al>2\al-3>-1$. 
	
		But, we also have to worry about the sum in $k: \sum_{k=2}^\infty |E_k(x)|$. The worst term (i.e., the fastest growing in $k$) in the estimate for $\int_\rone |(\eta^{k\al} \chi(\eta))'''| d\eta$ occurs, when all three derivatives fall on $\eta^{k\al}$, which grows like $k^3$. Indeed, note that as $k\geq 2$, we have that $(\eta^{k\al})'''=k\al (k\al-1)(k\al-2)\eta^{k\al-3}$, where $k\al-3\geq 2\al-3>-1$, so it is still integrable near zero. 
        In all, 
		$$
		E_k(x)\leq C (2^{k\al} 16^{-k\al} k^3)x^{-3}=C (8^{-k\al} k^3)x^{-3}
		$$
		It follows that 
		$$
		\left|\sum_{k=2}^\infty (-1)^k \f{(2\pi)^{k\al}}{2^{k+1}}\int_\rone |\xi|^{k\al} \chi(16 \xi)e^{2\pi i x\xi} d\xi\right|\leq C x^{-3} \sum_{k=2}^\infty \f{(2\pi)^{k\al}8^{-k\al} k^3}{2^{k+1}}=C_1 x^{-3}.
		$$
		Thus, it remains to analyze the second integral in $K_{\al,2}(x)$. We have, by \eqref{126}, 
			\begin{eqnarray*}
			 	\int_\rone   |\xi|^{\al}\chi(16 \xi) e^{2\pi i x\xi} d\xi &=&   \int_\rone   |\xi|^{\al}e^{2\pi i x\xi} d\xi+ \int_\rone   |\xi|^{\al}(1-\chi(16 \xi))
			e^{2\pi i x\xi} d\xi =\\
			&=&  -\f{\pi^{-\al}}{\sqrt{\pi}} \f{\Ga\left(\f{\al+1}{2}\right)}{\Ga\left(-\f{\al}{2}\right)}|x|^{-1-\al} +\sum_{j=5}^\infty  \int_\rone   |\xi|^{\al} \psi(2^{-j} \xi)
			e^{2\pi i x\xi} d\xi=  \\
			&=& -\f{\pi^{-\al}}{\sqrt{\pi}} \f{\Ga\left(\f{\al+1}{2}\right)}{\Ga\left(-\f{\al}{2}\right)}|x|^{-1-\al} +\sum_{j=5}^\infty 2^{j(1+\al)} \hat{\psi}_\al (2^j x).
		\end{eqnarray*}
		where $\psi_\al(\xi):=|\xi|^\al \psi(\xi)$. Clearly, the sum is an error term, since for $|x|>1$ and every $N$, 
		$$
		|\sum_{j=5}^\infty 2^{j(1+\al)} \hat{\psi}_\al (2^j x)|\leq C_N  \sum_{j=5}^\infty \f{2^{j(1+\al)}}{(2^j |x|)^N} \leq C_N |x|^{-N}.
		$$
		This establishes \eqref{g:20}, because one can deploy the formula $\Ga(z+1)=z\Ga(z)$ and rewrite 
		$$
		-\f{\pi^{-\al}}{\sqrt{\pi}} \f{\Ga\left(\f{\al+1}{2}\right)}{\Ga\left(-\f{\al}{2}\right)}=
		-\f{\pi^{-\al}}{\sqrt{\pi}} \f{\Ga\left(\f{\al-1}{2}\right)\f{\al-1}{2}}{\f{\Ga\left(1-\f{\al}{2}\right)}{-\f{\al}{2} }}=-
		\f{\pi^{-\al}\al(\al-1)}{4\sqrt{\pi}} \f{\Ga\left(\f{\al-1}{2}\right)}{\Ga\left(\f{2-\al}{2}\right)},
		$$
		as in \eqref{g:20}. Regarding $K_\al$, we use the formula 
		$$
		K_\al(x)=-\f{(2\pi)^\al}{4} \int_\rone   |\xi|^{\al}\chi(16 \xi) e^{2\pi i x\xi}d\xi =\f{2^{\al-4}\al(\al-1)}{\sqrt{\pi}} \f{\Ga\left(\f{\al-1}{2}\right)}{\Ga\left(\f{2-\al}{2}\right)} +O(x^{-3}).
		$$

	Finally, regarding \eqref{g:45}, we need estimates on the following $\int_0^\infty \f{\xi}{2+(2\pi \xi)^\al} \sin(2\pi x \xi) d\xi$. After three integrations by parts, we can estimate 
	$$
	|\int_0^\infty \f{\xi}{2+(2\pi \xi)^\al} \sin(2\pi x \xi) d\xi|\leq  \f{C}{x^3} \left(\int_{|\xi|<1} \xi^{\al-2} d\xi+
	\int_{|\xi|\geq 1} \xi^{-\al-2} d\xi\right)\leq C |x|^{-3}. 
	$$
	\subsection{Proof of Lemma \ref{g:11}}
	For the proof of Lemma \ref{g:11}, we have already discussed the positivity of the kernel $\tilde{K}_\al$. For the $L^\infty$ bound, we have 
	$$
	\sup_x |\tilde{K}_\al(x)|\leq \int_\rone \f{1}{2+4\pi^2 (1-c^2)\xi^2+ (2\pi |\xi|)^\al} d\xi <\infty.
	$$
	Assume that $x>1$. We consider the difference 
	\begin{eqnarray*}
		K_\al(x)-\tilde{K}_\al(x)=4\pi^2(1-c^2) \int_{\rone} \f{\xi^2}{(2+4\pi^2(1-c^2)\xi^2+(2\pi)^\al \xi^\al)(2+(2\pi)^\al \xi^\al)} \cos(2\pi \xi x) d\xi.
	\end{eqnarray*}
The expression on the right allows four times integration by parts, as 
$$
|\left(\f{\xi^2}{(2+4\pi^2(1-c^2)\xi^2+(2\pi)^\al \xi^\al)(2+(2\pi)^\al \xi^\al)}\right)^{''''}|\leq \f{C}{\xi^{2-\al}(1+\xi^{2+\al})},
$$
	which is integrable. So, 
	$$
		|K_\al(x)-\tilde{K}_\al(x)|\leq C (1+|x|)^{-4}, 
	$$
	well within the error range. It follows that the asymptotic for $\tilde{K}_\al$ are the same as those for $K_\al(x)$. 
	
	\subsection{Proof of Lemma \ref{g:12}}
	The proof of the bound \eqref{fr:60} proceeds similarly to the bound \eqref{g:20}. Indeed, we first split $K_\al=K_{\al,1}+K_{\al,2}$. After the realization that $K_{\al,2}$ has integrable derivatives of all orders (and hence part of the error term), we perform an identical power series expansion (as in the proof of  Lemma \ref{g:10}), which yields the representation \eqref{m:10}. The first function there is Schwartz, the main term is again, 
	\begin{equation}
		\label{fr:100} 
		 -\f{(2\pi)^{\al}}{4} 
		\int_\rone   |\xi|^{\al}\chi(16 \xi) e^{2\pi i x\xi} d\xi 
	\end{equation}
	while the term $\sum_{k=2}^\infty (-1)^k \f{(2\pi)^{k\al}}{2^{k+1}} \int_\rone |\xi|^{k\al} \chi(16 \xi)e^{2\pi i x\xi} d\xi$ again  is an error term.   Indeed, 
	we have, after reorganizing and five integration by parts, 
\begin{eqnarray*}
	|E_k(x)| &= & \left|\int_\rone |\xi|^{k\al} \chi(16 \xi)e^{2\pi i x\xi} d\xi\right|=
	16^{-k\al-1} 	\left|\int_\rone |\eta|^{k\al} \chi(\eta)e^{\pi i x\f{\eta}{8}} d\eta\right|\\
	&\leq &  C 16^{-k\al} 	\left|\int_0^\infty \eta^{k\al} \chi(\eta)\cos(\pi  x\f{\eta}{8})d\eta\right|\leq C \f{16^{-k\al}}{x^5} 
	\left|\int_0^\infty  (\eta^{k\al} \chi(\eta))''''' \sin(\pi  x\f{\eta}{8}) d\eta\right| \\
	&\leq & C \f{16^{-k\al}}{x^5} \int_\rone |(\eta^{k\al} \chi(\eta))'''''| d\eta\leq 
	C \f{8^{-k\al} k^5}{x^5},
\end{eqnarray*}
where we have used $||(\eta^{k\al} \chi(\eta))'''''|\leq C k^5 |\eta|^{k\al-5}$ and  $k\al-5\geq 2\al-5>-1$, which makes this function inegrable near zero.  Thus, 
$$
\sum_{k=2}^\infty |\f{(2\pi)^{k\al}}{2^{k+1}} E_k(x)|\leq C |x|^{-5}, 
$$
	well within the margin of error. 	It remains to extract the leading order term out of \eqref{fr:100}. 
	
	As in the case $\al\in (1,2)$, we have the formula
$$
		\int_\rone   |\xi|^{\al}\chi(16 \xi) e^{2\pi i x\xi} d\xi =  \f{\pi^{-\al}}{\sqrt{\pi}}  \f{\Ga\left(\f{\al+1}{2}\right)}{\Ga\left(-\f{\al}{2}\right)} +O(|x|^{-N}), 
$$
	Reworking the denominator, we obtain
	\begin{eqnarray*}
			K_\al(x) &=& -\f{(2\pi)^\al}{4} 	\int_\rone   |\xi|^{\al}\chi(16 \xi) e^{2\pi i x\xi} d\xi+O(|x|^{-5})= \\
			&=& -\f{2^{\al-5}\al(\al-1)(\al-2)}{\sqrt{\pi}} 
			\f{\Ga\left(\f{\al-1}{2}\right)}{\Ga\left(\f{4-\al}{2}\right)}|x|^{-1-\al} +O(|x|^{-5}).
	\end{eqnarray*}

	\section{Proof of Lemma \ref{le:45}}
	Define 
	$$
	X_0=\{f\in H^s(\rn): \dpr{f}{\phi'}=0\}.
	$$
	Consider a function 
	$$
	{\mathbb G}(w,v,\si)(t,x)=\phi(x+\si(t))-\phi(x)+v(t,x)-w(t,x).
	$$
	We would like to show that for every small $\|w\|_{H^s}<<1$, one can find $v=v_w\in X_0, \si=\si_w\in C$.
	
	To this end,  we show that $G:C^1([0,T], H^s)\times C([0,]), X_0)\times C^1[0,T]\to C([0,]), H^s)$. Indeed,  one can write 
	 $$
	 	{\mathbb G}(w,v,\si) = \si(t)\int_0^1 \phi'(x+\tau \si(t)) d\tau +v(t,x)-w(t,x)\in C([0,T], H^s). 
	 $$
	 Also ${\mathbb G}(0,0,0)=0$.  By the implicit function theorem, it suffices to prove that 
	 $$
	 d{\mathbb G}(0,0,0) (\tilde{v}, \tilde{\si})=\phi'(x) \tilde{\si}(t)+\tilde{v}(t,x)
	 $$
	is an isomorphism on $C^1([0,T], H^s)$. Specifically, for any $h\in C^1([0,T], H^s)$, we need to resolve the relation 
	\begin{equation}
		\label{700} 
			\phi'(x) \tilde{\si}(t)+\tilde{v}(t,x)=h(t,x).
	\end{equation}
	By taking inner product with $\phi'$ and taking $P_{\{\phi'\}^\perp}$ respectively,  we obtain the unique solution 
	$$
	\tilde{\si}(t)=\|\phi'\|^{-2} \dpr{h(t, \cdot)}{\phi'}, \ \ \tilde{v}=P_{\{\phi'\}^\perp}[h].
	$$
	Moreover, 
	\begin{eqnarray*}
	& & 		\max_{0\leq t\leq T}\|\tilde{\si}(t)\|\leq \|\phi'\|^{-1} \max_{0\leq t\leq T} \|h(t, \cdot)\|, 
		\max_{0\leq t\leq T}\|\tilde{\si'}(t)\|\leq \|\phi'\|^{-1} \max_{0\leq t\leq T} \|h'(t, \cdot)\| \\ 
		& & \max_{0\leq t\leq T} \|\tilde{v}(t, \cdot)\|_{H^s}\leq (1+ \|\phi'\|^{-1}) 
		 \max_{0\leq t\leq T} \|h(t, \cdot)\|_{H^s},  \\
		 & & \max_{0\leq t\leq T} \|\p_t \tilde{v}(t, \cdot)\|_{H^s}\leq (1+ \|\phi'\|^{-1}) 
		 \max_{0\leq t\leq T} \|\p_t h(t, \cdot)\|_{H^s}.
	\end{eqnarray*}
Conversely, for any $\tilde{\si}\in  C^1[0,T], \tilde{v}\in C^1([0,T], H^s)$, we obtain a unique $h$ through the formula \eqref{700}, and also 
$$
\|h\|_{ C^1([0,T], H^s)}\leq \|\phi\| \|\tilde{\si}\|_{C^1[0,T]}+ \|\tilde{v}\|_{C^1([0,T], X_0)},
$$
	as required.

 \bibliographystyle{plain}  
\bibliography{refs1}  

  

\end{document}